\documentclass[12pt,a4paper]{amsart}
\usepackage[utf8]{inputenc}
\usepackage[T1]{fontenc}
\usepackage[english]{babel}
\usepackage[all]{xy}
\usepackage{fancyhdr}
\usepackage{array}
\usepackage{amsfonts}
\usepackage{bbm}
\usepackage{amsmath}
\usepackage{amssymb}
\usepackage{graphicx}
\usepackage{stmaryrd}
\usepackage{moreverb}
\usepackage{geometry}
\usepackage{amsthm}
\usepackage{color}
\usepackage{bbold}
\usepackage[style = alphabetic, maxbibnames = 99, backend = bibtex]{biblatex}
\usepackage{yhmath}
\usepackage{tikz}
\usepackage{hyperref}
\usepackage{cleveref}
\usepackage{csquotes}

\title[Klt varieties with semipositive anticanonical bundle]{Klt varieties with semipositive anticanonical bundle: symmetric tensors, holonomy and structure}
\author{Louis Dailly}
\address{Louis Dailly, Institut de Math\'ematiques de Toulouse, Universit\'e Paul Sabatier, 31062 Toulouse Cedex~9, France}
\email{\href{mailto:louis.dailly@utoulouse.fr}{louis.dailly@utoulouse.fr}}
\urladdr{\href{https://www.louis-dailly.github.io}{louis-dailly.github.io/}}

\newcommand{\incl}[1][r]
  {\ar@<-0.2pc>@{^(-}[#1] \ar@<+0.2pc>@{-}[#1]}
\newcommand{\immouv}[1][r]
   {\ar@{}[#1] |*[o][F]{\hbox{%
         \vrule width 1.5mm height 0pt depth 0pt%
         \vrule width 0pt height .75mm depth .75mm%
         }}
     \ar@{^{(}->}[#1]}

\newcommand{\Z}{\mathbb{Z}}
\newcommand{\Q}{\mathbb{Q}}
\newcommand{\R}{\mathbb{R}}
\newcommand{\C}{\mathbb{C}}

\newcommand{\reg}{\mathrm{reg}}

\newcommand{\End}{\mathrm{End}\,}
\newcommand{\GL}{\mathrm{GL}\,}
\newcommand{\PGL}{\mathrm{PGL}\,}
\newcommand{\SU}{\mathrm{SU}\,}
\newcommand{\U}{\mathrm{U}\,}
\newcommand{\Sp}{\mathrm{Sp}\,}
\newcommand{\id}{\mathrm{id}}

\newcommand{\ddc}{\mathrm{dd^c}}
\newcommand{\Ric}{\mathrm{Ric}\,}
\newcommand{\tr}{\mathrm{tr}\,}
\newcommand{\Sym}{\mathrm{Sym}\,}
\newcommand{\Hol}{\mathrm{Hol}}
\newcommand{\hol}{\mathrm{hol}}

\newcommand{\Aut}{\mathrm{Aut}\,}

\def\dar[#1]{\ar@<2pt>[#1]\ar@<-2pt>[#1]}
  \entrymodifiers={!!<0pt,0.7ex>+}

\theoremstyle{plain}
\newtheorem{theo}{Theorem}[section]
\newtheorem{theoalph}{Theorem}

\newtheorem{lemm}[theo]{Lemma}
\newtheorem{prop}[theo]{Proposition}

\theoremstyle{definition}
\newtheorem{defi}[theo]{Definition}

\newtheorem{rema}[theo]{Remark}

\newtheorem*{ackn}{Acknowledgement}
\newtheorem*{aid}{AI Disclosure}

\newenvironment{itemize*}%  % D\'efinition d'un nouvel environnement
    {\begin{itemize}%
      \setlength{\itemsep}{0pt}%
      \setlength{\parskip}{0pt}}%
    {\end{itemize}}

\bibliography{mabiblio}
\AtBeginBibliography{\tiny}

\begin{document}

\begin{abstract}
	The aim of this article is to study klt projective varieties $X$ such that $-K_X$ is semipositive from the differential-geometric point of view. The study of the holonomy of such varieties yields generalizations of several results. We show that the dimension of the space of symmetric differentials is bounded by a constant depending on the augmented irregularity of $X$, generalizing a result of Kobayashi. If we endow $X$ with a metric $\omega$ such that $\Ric \omega \geq 0$, then we show that the holonomy of $(X_\reg, \omega_\reg)$ is compact. If we assume that $X$ has canonical singularities, then we show that the general fiber of the Albanese map can be decomposed into a product of irreducible Calabi--Yau varieties, irreducible holomorphic symplectic varieties and hereditarily rationally connected varieties by the holonomy generalizing a structure theorem of Campana, Demailly and Peternell. To conclude, we show that the fundamental group of the regular locus of such varieties is virtually abelian.
\end{abstract}

\maketitle

\tableofcontents

\section{Introduction}

Let $X$ be a complex projective manifold. When $X$ is simply connected and has a vanishing first Chern class, it has been shown by Kobayashi in \cite{Kob80} that $X$ have neither non-trivial symmetric tensor fields nor non-trivial symmetric differentials. In the same article, the author shows an upper bound of the dimension of holomorphic symmetric differentials (resp. tensor fields) on $X$, when $X$ has a semipositive (resp. seminegative) first Chern class.

\begin{theo}\textup{\cite[Thm.~7]{Kob80}}
    Let $X$ be a compact Kähler manifold with $c_1(X) \geq 0$ (resp. $c_1(X) \leq 0$). Let
$\widetilde{X}$ be the universal cover of $X$ and let $k$ be the largest integer such that $\widetilde{X}$ is biholomorphic to $\C^k \times Y$ , where $Y$ is a complex manifold. Then, we have:
$$
h^0(X, \Sym^m \Omega^1_X) \leq \binom{m+k-1}{m} \; \text{(resp. }h^0(X, \Sym^m T_X) \leq \binom{m+k-1}{m}).
$$
\end{theo}

This upper bound depends on the dimension of the euclidean factor $k$ in the universal cover. The proof of these results relies on techniques from differential geometry, namely the Bochner's principle and the study of the restricted holonomy of $X$ endowed with a suitable metric $\omega$, coupled with the deep De Rham decomposition theorem. 
In the non-negative curvature case, the authors of \cite{CDP14} establish a link between the differential and algebraic techniques through a structure theorem providing a decomposition of the universal cover of $X$.  Such a decomposition theorem generalizes the classical Beauville--Bogomolov decomposition theorem to the case of manifolds with semi-positive first Chern class. Thanks to this decomposition theorem, they show that the Albanese map and the MRC-fibration are \emph{locally constant} (it means that it is a locally trivial fibration, obtained from a representation of the base space with values in the automorphism group of the fiber, see \cite{MW21} for details). The local constancy of the Albanese map has been extended to the case where $-K_X$ is nef in \cite{Cao19}. In the same article, the author gives a new interpretation of the quantity $k$ aforementioned: it is equal to \emph{the augmented irregularity} of $X$, denoted $\widetilde{q}(X)$:
$$
\widetilde{q}(X) := \sup \left\lbrace h^0(X', \Omega^1_{X'}) \;\;|\;\; \pi : X' \to X \text{ finite \'{e}tale cover } \right\rbrace.
$$
This gives a more algebraic and more intrinsic interpretation of $k$: this quantity is for instance independent of the metric $\omega$, and depends only on finite \'{e}tale covers that are encapsulated in algebraic invariant like the algebraic fundamental group of $X$. The proof of this result is based on algebraic structure theorems of varieties with nef anticanonical bundle, namely the local constancy of the Albanese map proved in the same article.\\

In the singular setting, the authors of \cite{GGK19} establish a generalization of many results when $c_1(X) = 0$. Coupled with \cite{Dru18, HP19}, they show a singular version of Beauville--Bogomolov decomposition, and identify irreducible Calabi--Yau varieties and irreducible holomorphic symplectic varieties via holonomy. In their study, they prove a version of Cheeger--Gromoll theorem about compactness of holonomy groups (see \cite[Lemma.~7.8]{GGK19} and \cite[Thm.~6]{CG71}).\\

The aim of this article is to investigate a mildly singular version of these results. We first state an estimate of dimension of reflexive symmetric tensors:

\begin{theoalph}\label{Bound of symmetric tensors}
	Let $X$ be klt projective variety such that $-K_X$ is semi-positive (resp. semi-negative). Let $k$ be the rank of the subspace fixed by the restricted holonomy. Then, we have:
	$$
	h^0(X, \Sym^{[m]} \Omega^{[1]}_X) \leq \binom{m+k-1}{m} \; \text{(resp. }h^0(X, \Sym^{[m]} \mathcal{T}_X) \leq 	\binom{m+k-1}{m}).
	$$
\end{theoalph}

\noindent
\emph{Strategy of the proof}\\
The proof in the smooth case goes in three steps:
\begin{itemize}
	\item Show Bochner's principle for symmetric differentials and vanishing when we apply Ricci tensor
	\item Study of the holonomy representation on the universal cover
	\item Show that global sections on the universal cover can come only from the euclidean factor in the De Rham splitting theorem.
\end{itemize}

To prove the first step, we show a sheaf-theoretic version of the Bochner principle: the subsheaves of $\Sym^{[m]}\Omega^{[1]}_X$ having a vanishing slope are parallel, they provide directions along which the Ricci tensor vanishes, and they are polystable. We prove this property by using approximated metrics on a resolution of $X$ to estimate how second fundamental forms behave on the regular locus of $X$, see e.g. \cite{Eno88, Gue16, GGK19, GT22, DGP24}. The second step is more subtle: one can define the restricted holonomy in the singular context by considering the holonomy of the metric on the regular part. However, since $(X_\reg, \omega_\reg)$ is not a complete Riemannian manifold, the De Rham splitting theorem fails, and the decomposition of the representation coming from restricted holonomy induced only a decomposition of the variety on small neighborhoods. To overpass this issue, it is possible to construct finite quasi-\'{e}tale covers (called \emph{weak holonomy covers}) on which the decomposition of irreducible representations of the holonomy acting on one point extend to a decomposition of the tangent sheaf $\mathcal{T}_X$. The third step is obtained similarly to \cite{Kob80}.\\

In the smooth case, we observe that the rank of the subspace fixed by the restricted holonomy is equal to the dimension of the euclidean factor $k$. This quantity remains equal to the augmented irregularity of $X$ in the singular setting.

\begin{theoalph}\label{Identification euclidean factor augmented irregularity}
	Let $X$ be a klt projective variety such that $-K_X$ is semipositive. Then, $k = \widetilde{q}(X)$.
\end{theoalph}

Above a regular point $p$ of $X$, we can consider the subspaces of tangent vectors of $T_pX$ that are fixed by the (resp. restricted) holonomy, denoted $E_0$ (resp. $F_0$). Since the restricted holonomy is included in the total holonomy, we have $E_0 \subseteq F_0$. When we consider a quasi-\'{e}tale cover, one can show that the dimension of $F_0$ is unchanged, when the dimension of $E_0$ increases. This theorem establishes that there is a finite cover on which $E_0 = F_0$. To prove this, we show that flat directions in the tangent sheaves come from the Albanese map. That means that the subsheaf of $\mathcal{T}_X$ induced by fixed vectors coincides with $\mathrm{alb}^*\mathcal{T}_{\mathrm{Alb}(X)}$ using local constancy of MRC-fibrations and Albanese morphisms given by \cite{MW21}.\\

As pointed out in \cite{GGK19}, the holonomy of $(X_\reg, \omega_\reg)$ is connected up to a finite quasi-\'{e}tale cover of $X$ when $c_1(X) = 0$. This property remains true when $-K_X$ is semipositive.

\begin{theoalph}\label{Compactness of holonomy}
	Let $X$ be a klt projective variety endowed with a Kähler metric $\omega$ such that $\Ric \omega \geq 0$ is semipositive, and $p \in X_\reg$. Then, $\Hol_p(X_\reg, \omega_\reg)$ is compact.
\end{theoalph}

\noindent
\emph{Strategy of the proof}\\
In the case $c_1(X) = 0$, the existence of a finite quasi-\'{e}tale cover on which the variety splits as $A \times Z$ where $A$ is an abelian variety and $Z$ a variety such that $\widetilde{q}(Z) = 0$ allow the author to study separately the holonomy of each factor, so that it suffices to show that the holonomy of $Z$ is compact. This is managed by topological considerations, since the difference between the total holonomy and the restricted one is encapsulated in the representation $\pi_1(Z_\reg) \to \Hol(Z_\reg)/\Hol(Z_\reg)^\circ$. In our case, the result of \cite{MW21} shows that we do not have such a splitting but a locally constant fibration coming from the Albanese morphism $X \to A$, given by a representation $M : \pi_1(A) \to \Aut(X_a)$. We can study the holonomy of the general fiber $X_a$ of this fibration. However, it is not sufficient to understand the holonomy of the whole $X$ and the main difficulty is to understand the extension $\Hol(X_\reg)/\Hol(X_{a \reg})$. We show that this extension is fully determined by $M$. Moreover, we show that the image of $M$ preserves the induced metric on $X_a$, and thanks to compactness of this group (this is analogous to classical compactness of isometry groups) we deduce the compactness of $\Hol(X_\reg)$.\\

In the structure of klt varieties with semipositive anticanonical bundle, the rationally connected varieties play an important role, since they are positively curved irreducible pieces arising from the holonomy decomposition (see \cite{CDP14}). In the singular setting, rational connectedness can also appear through quotients of K-trivial varieties (see \emph{e.g.} \cite{OT15} for examples of finite quotients of abelian varieties that are rationally connected). In the holonomy decomposition, the varieties with total holonomy $\U(n)$ or non-flat locally symmetric varieties are \emph{hereditarily rationally connected}: for any finite quasi-\'{e}tale morphism $X' \to X$, the variety $X'$ is rationally connected. To conclude, we show a singular structure theorem generalizing the result of \cite{CDP14}.

\begin{theoalph}\label{Decomposition}
	Let $X$ be a projective variety with canonical singularities and such that $-K_X$ semipositive. Then, there exists a finite quasi-\'{e}tale morphism $X' \to X$ such that the Albanese morphism:
	$$
	\mathrm{alb}_{X'} : X' \to \mathrm{Alb}(X')
	$$
	is locally constant and the general fiber is a product $\prod_{j} Y_j \times \prod_{k} S_k \times  \prod_{l} Z_l$, where $Y_j$ are irreducible Calabi-Yau, $S_k$ are irreducible holomorphic symplectic and $Z_l$ are hereditarily rationally connected.
\end{theoalph}

\noindent
\emph{Strategy of the proof}\\
The main issue to adapt arguments from \cite{CDP14} is the lack of singular version of de Rham decomposition theorem. However, one can implement the strategy of the singular Beauville--Bogomolov decomposition theorem:
\begin{itemize}
    \item Show that $\mathcal{T}_X$ can be decomposed in a direct sum of foliations by the holonomy \cite{GGK19},
    \item Show that these holonomy foliations have algebraic leaves \cite{HP19},
    \item Show that decomposition of the tangent sheaf into algebraic integrable foliations induced a decomposition on a quasi-\'{e}tale cover \cite{Dru18}.
\end{itemize}
We implement this strategy to show that the decomposition can be achieved for the general fiber of the Albanese fibration of $X$. We show the first two points by using our Bochner's principle and considerations about the slopes and the pseudo-effectivity of sheaves  \cite{HP19, DGP24}. In order to get rid of the smoothness assumption in codimension 2 of \cite{HP19}, we use the existence of a $\Q$-factorial terminalization ensured by \cite{BCHM10}. Hence, we use Druel's integrability result \cite{Dru18} to show that the fiber can split on a quasi-\'{e}tale cover. It remains to show that there exists a finite quasi-\'{e}tale morphism above $X$ realizing this splitting of the fiber, this can be achieved by using group-theoretic arguments and the special structure of the fundamental group $\pi_1(X_\reg)$ coming from the local triviality of the Albanese map.\\

We conclude by considerations on fundamental groups of the regular locus of such varieties. Motivated by a singular version of Cheeger--Gromoll theorem, the recent results of \cite{FGSW26} prove that singular ICY and IHS varieties have a regular locus with finite fundamental groups. Using this we show that the varieties with semipositive anticanonical bundle are virtually abelian. 

\begin{theoalph}\label{Virtually abelian}
	Let $X$ be a projective variety with canonical singularities such that $-K_X$ semipositive. Then, $\pi_1(X_\reg)$ is virtually abelian.
\end{theoalph}

\noindent
\emph{Strategy of the proof}\\
From Theorem~\ref{Decomposition} and the aforementioned result, it suffices to show that hereditarily rationally connected pieces have finite fundamental groups. In the smooth case, this can be proved through the existence of a very free curve. If $X$ is such a variety, the authors of \cite{JLR26} show that the existence of a very free curve implies the finiteness of $\pi_1(X_\reg)$. Moreover, they give a criterion for the existence of such a curve, based on positivity properties of tangent sheaves. We show that this criterion is satisfied thanks to the positivity assumption on $-K_X$.

\begin{ackn}
    The author would like to thank Henri Guenancia for its constant support through fruitful discussions.
\end{ackn}
\begin{aid}
	The author used ChatGPT 5.5 Instant to familiarize with finite-index subgroups arguments to show Lemma~\ref{Group argument}. This use had not relevant impact on the present paper: the mathematical content and the exposition are due to the author.
\end{aid}

\section{Bochner principle}

\subsection{Setup and Notation}

Let $X$ be a klt projective variety of dimension $n$, we denote by $H_0$ an ample divisor on $X$. Hence, $K_X$ is $\Q$-Cartier, so that there exists an natural number $m$ such that $mK_X$ is Cartier. We assume that $-K_X$ is semi-positive, that means that there exists a smooth metric $h_{K_X}^m$ on $mK_X$ such that the smooth $(1,1)$-form:
$$
\alpha := \frac{1}{m} \textbf{i}\Theta(-m K_X, {h_{K_X}^m}^\vee) \in \left[-K_X \right]$$
is semi-positive. Let $\theta$ be a Kähler $(1,1)$-form representing $\left[ H_0 \right]$ and we denote by $\mu_{h_{K_X}^m}$ the adapted volume. From \cite[Thm.~6.3]{EGZ09}, there exists a unique potential $\Phi \in L^\infty(X)$ such that:
$$
\left( \theta + \ddc \Phi \right)^n = c \mu_{h_{K_X}^m}
$$
where $c := \frac{\left[ H_0 \right]^n}{\mu_{h_{K_X}^m}(X)}$.
We denote in the following $\omega := \theta + \ddc \Phi$, we have $\Ric \omega = \alpha$.

 On the regular locus $X_\reg$ of $X$, the $(1,1)$-form $\omega$ provides us with a K\"{a}hler metric, whose the associated hermitian metric will be denoted $h$ and the associated Chern connection is denoted $D_\omega$.

If $\beta$ is any $(0,1)$-form on $X_\reg$, we can construct $\sharp \beta \in \mathcal{C}^\infty(X_\reg, \mathcal{T}_{X_\reg})$ as the unique element such that we have for any element $v \in T^{0,1}X$:
$$
h(\sharp \beta, v) = \beta(v),
$$ 
and extend it to $(1,1)$-forms. In this way, for any $(1,1)$-form $\gamma$, the previous construction provides us with $\sharp \gamma$, a smooth section of $\End(\mathcal{T}_X)$ satsfying the following equality:
$$
n \gamma \wedge \omega^{n-1} = \tr_\End(\sharp \gamma) \omega^n.
$$

If $V$ is a vector space and $f : V \to V$ is an endomorphism, we denote $f^{\Sym^m} : \Sym^m V \to  \Sym^m V$ the endormophism defined on tensors by: 
$$
f^{\Sym^m}(v_1, \otimes \cdots \otimes v_m) = \sum_{1 \leq k \leq m} v_1 \otimes \cdots v_{k-1} \otimes f(v_k) \otimes v_{k+1} \otimes \cdots \otimes v_m.
$$
One can observe that: $\tr(f^{\Sym^m}) = \binom{n+m-1}{n} \tr(f)$. This operator is useful to express the curvature of metrics induced on symmetric powers of a bundle.

\subsection{Bochner principle for symmetric differentials}

In this section we show the following Bochner's principle for proper subsheaves of vanishing slope.

\begin{prop}\label{Polystability of vanishing slopes}
    Let $X$ be a klt projective variety such that $-K_X$ is semipositive. Let $\mathcal{F}$ be a proper subsheaf of $\Sym^{[m]} \Omega^{[1]}_X$ such that $\mu_{\omega}(\mathcal{F}) = 0$. Then,\\
	$(i)$ the sheaf $\mathcal{F}$ is parallel with respect to the connection induced by $\omega$, and $$\mathrm{pr}_{\mathcal{F}}(\overline{\sharp \alpha}^{\Sym^m}_{|\mathcal{F}}) = 0,$$
	$(ii)$ there exists a subsheaf $\mathcal{G}$ of $\Sym^{[m]}\Omega^{[1]}_X$ such that $\Sym^{[m]}\Omega^{[1]}_X = \mathcal{F} \oplus \mathcal{G}$, and the decomposition is orthogonal with respect to $\omega$,\\
	$(iii)$ $\mathcal{F}$ is polystable with respect to $\left[ \omega \right]$.
\end{prop}

\begin{proof}

\noindent
\textbf{Step 1 : Approximated metrics on a resolution}\\
Let $\pi : \widehat{X} \to X$ be a resolution of $X$, so that:
$$
K_{\widehat{X}} = \pi^*K_X + \sum_{i} a_i E_i
$$
with $a_i > -1$. Let $\widehat{\omega}$ be a K\"{a}hler metric on $\widehat{X}$. We denote $E := \sum_i E_i$ the reduced exceptional divisor. On the line bundle $\mathcal{O}(E_i)$, we consider a hermitian metric $h_i$ and a holomorphic section $\sigma_i$ such that $\mathrm{div}\, \sigma_i = E_i$, and we denote $\theta_i := \textbf{i}\Theta(E_i, h_i)$. Hence, one can define a volume form $\mathrm{d}V$ on $\widehat{X}$ such that;
\begin{equation}\label{Ricci of dV}
	\Ric \mathrm{d}V = \pi^*\alpha - \sum_i a_i \theta_i.
\end{equation}

Then, by application of Yau theorem, we can consider the potentials $\varphi_{t, \varepsilon} \in \mathcal{C}^\infty(\widehat{X}) \cap \mathrm{PSH}\,(\widehat{X}, \pi^*\theta + t\widehat{\omega})$:
\begin{equation}\label{Perturbed Monge--Ampere}
    (\pi^*\theta + t \widehat{\omega} + \ddc \varphi_{t, \varepsilon})^n = c_{t, \varepsilon} \prod_i (|\sigma_i|^2 + \varepsilon^2)^{a_i} \mathrm{d}V,
\end{equation}
where $c_{t, \varepsilon} := \frac{\left[ \pi^*\theta + t\widehat{\omega} \right]^n}{\int_{\widehat{X}} \prod_i (|\sigma_i| + \varepsilon^2)^{a_i} \mathrm{d}V }$ is a normalizing constant.
When $t, \varepsilon$ go to zero, the family $(\varphi_{t, \varepsilon})_{t, \varepsilon}$ converges weakly to $\varphi := \pi^*\Phi$ and locally smoothly on $\widehat{X} \backslash |E|$.
We denote in the following $\omega_{t, \varepsilon} := \pi^*\theta + t \widehat{\omega} + \ddc \varphi_{t, \varepsilon}$ and $h_{t, \varepsilon}$ the associated metric on $T_{\widehat{X}}$. On the line bundles $\mathcal{O}(E_i)$ and $\mathcal{O}(E)$, there exists also approximated metrics:
$$
h_{i, \varepsilon} := (|\sigma_i|^2 + \varepsilon^2)^{-1} h_i, \qquad h_{E, \varepsilon} := \prod_i h_{i, \varepsilon}.
$$
Hence, the curvature tensor associated to $h_{i, \varepsilon}$ is:
\begin{align}
    \theta_{i, \varepsilon} & := \textbf{i}\Theta(E_i, h_{i, \varepsilon}) = \ddc \log \left( |\sigma_i|^2 + \varepsilon^2 \right) + \textbf{i}\Theta(E_i, h_i)\\
    & = \underbrace{\frac{\varepsilon^2}{|\sigma_i|^2 + \varepsilon^2} \theta_i}_{=: \gamma_{i, \varepsilon}} + \underbrace{\frac{\varepsilon^2  |D\sigma_i|^2}{(|\sigma_i|^2 + \varepsilon^2)^2}}_{=: \beta_{i, \varepsilon}}.\label{Curvature of perturbed metrics on E_i}
\end{align}
We deduce from \eqref{Ricci of dV} and \eqref{Perturbed Monge--Ampere}:
\begin{equation}\label{Ricci of perturbed metrics}
	\Ric \omega_{t, \varepsilon} = \pi^*\alpha - \sum_i a_i \theta_{i, \varepsilon}.
\end{equation}

\noindent
\textbf{Step 2 : Curvature of the subsheaf}\\
Let $\mathcal{F}$ be a subsheaf of $\Sym^{[m]}\Omega^{[1]}_X$ of rank $r$. Up to consider its saturation, let us assume that $\mathcal{F}$ is saturated. In the previous step, we have constructed approximated metrics $\omega_{t, \varepsilon}$ on a resolution $\pi : \widehat{X} \to X$. Then, one can define $\widehat{\mathcal{F}}$ as the saturation of $\mathcal{F}$ in $\Omega^1_{\widehat{X}}$, so that we have:
$$
c_1(\widehat{\mathcal{F}}) \cdot \left[ \pi^*\omega \right]^{n-1} = c_1(\mathcal{F}) \cdot \left[ \omega \right]^{n-1}
$$ (see the proof of \cite[Lemma.~4.3]{Dai26} for the construction of $\widehat{\mathcal{F}}$ and the equality of slopes). Now, there exists a big open subset $\widehat{U}$ such that $\widehat{\mathcal{F}}_{|\widehat{U}}$ is locally free and it is a subbundle of $\Omega^1_{\widehat{X}}$. We have by \cite[Proof of $(\ast\ast)$, p.166]{Kob80} that:
$$
c_1(\widehat{\mathcal{F}}) \cdot \left[ \omega_{t, \varepsilon} \right]^{n-1} = \int_{\widehat{U}} c_1(\widehat{\mathcal{F}}, h_{t, \varepsilon}) \wedge \omega_{t, \varepsilon}^{n-1}.
$$
We have above $\widehat{U}$:
\begin{align}
     - \frac{1}{n} \overline{\sharp \Ric \omega_{t, \varepsilon}}^{\Sym^m} \omega_{t, \varepsilon}^n & = \textbf{i} \Theta(\Omega^1_{\widehat{X}}, h_{t,\varepsilon}^\vee)^{\Sym^m} \wedge \omega_{t, \varepsilon}^{n-1} \\
     & = 
\begin{bmatrix}
    \textbf{i} \Theta(\widehat{\mathcal{F}}, \Sym^m h^\vee_{t, \varepsilon |\widehat{\mathcal{F}}}) \wedge \omega_{t, \varepsilon}^{n-1} - \beta_{t, \varepsilon}^* \wedge \beta_{t, \varepsilon} \wedge \omega_{t, \varepsilon}^{n-1} & \ast \\
    \ast & \ast 
\end{bmatrix}
\end{align}
where $\beta_{t, \varepsilon}$ is the second fundamental form associated to the subbundle $\widehat{\mathcal{F}}$. We write:
$$
\textbf{i} \Theta(\widehat{\mathcal{F}}, \Sym^m h^\vee_{t, \varepsilon |\widehat{\mathcal{F}}}) \wedge \omega_{t, \varepsilon}^{n-1} = A_{t, \varepsilon} \omega_{t, \varepsilon}^n \qquad \text{and} \qquad -\beta_{t, \varepsilon}^* \wedge \beta_{t, \varepsilon} \wedge \omega_{t, \varepsilon}^{n-1} = P_{t, \varepsilon} \omega_{t, \varepsilon}^{n},
$$
where $A_{t, \varepsilon}, P_{t, \varepsilon}$ are symmetric endomoprhisms with $P_{t, \varepsilon}$ semi-positive. Using \eqref{Ricci of perturbed metrics}, we deduce that:
$$
A_{t, \varepsilon} + P_{t, \varepsilon} = \frac{1}{n} \left( - \mathrm{pr}_{\widehat{\mathcal{F}}} \overline{\sharp \pi^*\alpha}^{\Sym^m}_{| \widehat{\mathcal{F}}} + \sum_i a_i \mathrm{pr}_{\widehat{\mathcal{F}}} \overline{\sharp \theta_{i, \varepsilon}}^{\Sym^m}_{| \widehat{\mathcal{F}}} \right).
$$

\noindent
\textbf{Step 3 : Control of error terms}\\
The goal of this step is to show the convergence:
$$
\int_{\widehat{X}} \tr_\End(\mathrm{pr}_{\widehat{\mathcal{F}}} \overline{\sharp \theta_{i, \varepsilon}}^{\Sym^m}_{| \widehat{\mathcal{F}}}) \omega_{t, \varepsilon}^n \underset{t, \varepsilon \to 0}{\longrightarrow} 0
$$

We recall from \eqref{Curvature of perturbed metrics on E_i} that $\theta_{i, \varepsilon} = \gamma_{i, \varepsilon} + \beta_{i, \varepsilon}$. There exists $C > 0$ such that $\pm \theta_i \leq C \widehat{\omega}$. Hence, because the operators $\sharp \bullet$, $\overline{\bullet}$ $\bullet^{\Sym^m}$ and $\mathrm{pr}_{\widehat{\mathcal{F}}} \bullet_{|\widehat{\mathcal{F}}}$ preserve positivity, the following holds:
\begin{align}
    \pm \tr_\End \left(\mathrm{pr}_{\widehat{\mathcal{F}}} \overline{\sharp \gamma_{i, \varepsilon}}^{\Sym^m}_{|\widehat{\mathcal{F}}} \right)
    & \leq \frac{C \varepsilon^2}{|\sigma_i|^2 + \varepsilon^2} \tr_\End \left( \mathrm{pr}_{\widehat{\mathcal{F}}} \overline{\sharp \widehat{\omega}}^{\Sym^m}_{|\widehat{\mathcal{F}}} \right)\\
    & \leq \frac{C \varepsilon^2}{|\sigma_i|^2 + \varepsilon^2} \tr_\End \left( \overline{\sharp \widehat{\omega}}^{\Sym^m} \right) \\
    & \leq C \binom{n+m-1}{n} \frac{\varepsilon^2}{|\sigma_i|^2 + \varepsilon^2} \tr_\End \left( \sharp \widehat{\omega} \right).\label{Bound on gamma}
\end{align}
Moreover, we have:
\begin{align}
    0 \leq \tr_\End( \mathrm{pr}_{\widehat{\mathcal{F}}} \overline{\sharp \beta_{i, \varepsilon}}^{\Sym^m}_{|\widehat{\mathcal{F}}} )
    & \leq \tr_\End( \overline{\sharp \beta_{i, \varepsilon}}^{\Sym^m} ) \\
    & \leq \binom{n+m-1}{n}  \tr_\End \left( \sharp \beta_{i, \varepsilon} \right).\label{Bound on beta}
\end{align}
It follows that:
\begin{align}
    \left| \int_{\widehat{X}} \tr_\End(\mathrm{pr}_{\widehat{\mathcal{F}}} \overline{\sharp \theta_{i, \varepsilon}}^{\Sym^m}_{| \widehat{\mathcal{F}}}) \omega_{t, \varepsilon}^n \right| \leq C'_{n,m} \int_{\widehat{X}} \frac{\varepsilon^2}{|\sigma_i|^2 + \varepsilon^2} \widehat{\omega} & \wedge \omega_{t, \varepsilon}^{n-1} \\
    & + C_{n,m} \int_{\widehat{X}} \beta_{i, \varepsilon} \wedge \omega_{t, \varepsilon}^{n-1},
\end{align}
where $C_{n,m} = \frac{1}{n}\binom{n+m-1}{n}$ and $C'_{n,m} = C_{n,m}C$. It remains to show that these integrals go to zero when $t, \varepsilon$ go to zero. We have:
\begin{itemize}
    \item $\int_{\widehat{X}} \frac{\varepsilon^2}{|\sigma_i|^2 + \varepsilon^2} \widehat{\omega} \wedge \omega_{t, \varepsilon}^{n-1} \to 0$ by \cite[Claim~9.5]{GGK19},\\
    \item $\int_{\widehat{X}} \beta_{i, \varepsilon} \wedge \omega_{t, \varepsilon}^{n-1} = \int_{\widehat{X}} \theta_{i, \varepsilon} \wedge \omega_{t, \varepsilon}^{n-1} - \int_{\widehat{X}} \gamma_{i, \varepsilon} \wedge \omega_{t, \varepsilon}^{n-1}$
    and the first integral in the right hand side converge to zero since it is equal to $\left[ E_i \right] \cdot \left[ \pi^*\theta + t\widehat{\omega} \right]$ and $E_i$ is $\pi$-exceptional, and the second integral go to zero by \eqref{Bound on gamma} and the first point with $m = 1$.
\end{itemize}

\noindent
\textbf{Step 4 : Conclusion}\\
By Steps 2 and 3 and Fatou lemma, we have:
\begin{align*}
    \mu_\omega(\mathcal{F}) & = \frac{1}{r}\limsup_{t, \varepsilon \to 0} \int_{\widehat{U}} \tr_\End(A_{t, \varepsilon}) \omega_{t, \varepsilon}^{n} \\
    & = \frac{1}{r} \limsup_{t, \varepsilon \to 0} \int_{\widehat{U}} - \tr_\End(P_{t, \varepsilon} + \mathrm{pr}_{\widehat{\mathcal{F}}} \overline{\sharp \pi^*\alpha}^{\Sym^m}_{| \widehat{\mathcal{F}}}) \omega_{t, \varepsilon}^n \\
    & \leq \frac{1}{r} \int_{\widehat{U}} \limsup_{t, \varepsilon \to 0} - \tr_\End(P_{t, \varepsilon} + \mathrm{pr}_{\widehat{\mathcal{F}}} \overline{\sharp \pi^*\alpha}^{\Sym^m}_{| \widehat{\mathcal{F}}}) \omega_{t, \varepsilon}^n
\end{align*}

Let $\beta_\infty$ be the second fundamental form of $\widehat{\mathcal{F}}$. Since $\beta_{t, \varepsilon}$ converges to $\beta_\infty$ locally smoothly over $\widehat{U} \backslash E$, if we denote by $-\beta_\infty^* \wedge \beta_\infty \wedge \pi^*\omega^{n-1} = P_\infty \pi^*\omega^n$ over $\widehat{U} \backslash E$, we have $P_\infty$ that is a symmetric semi-positive endomorphism and the latter integral is equal to:
$$
    \int_{\widehat{U} \backslash E} -\tr_\End(P_\infty + \mathrm{pr}_{\widehat{\mathcal{F}}} \overline{\sharp \pi^*\alpha}_{| \widehat{\mathcal{F}}}) \pi^*\omega^n.
$$
Since the integrand is non-positive, we deduce that any subsheaf of $\Sym^{[m]}\Omega^{[1]}_X$ has non-positive slope. From now, we assume that $\mu_\omega(\mathcal{F}) = 0$ (it means that $\mathcal{F}$ is a destabilizing subsheaf). We get that $P_\infty = 0$ and $\mathrm{pr}_{\widehat{\mathcal{F}}} \overline{\sharp \pi^*\alpha}^{\Sym^m}_{| \widehat{\mathcal{F}}} = 0$. Hence, we have $\beta_\infty = 0$, and the splitting $\widehat{\mathcal{F}} \oplus \widehat{\mathcal{F}}^\perp$ is holomorphic over $\widehat{U} \backslash E$, and $\widehat{\mathcal{F}}$ is parallel with respect to $\pi^*\omega$.\\

Now, since $E$ is $\pi$-exceptional, we have by projecting the decomposition over the big open subset $U := \pi(\widehat{U}) \cap X_\reg$ (its complement is an analytic subset of $X$ of codimension at least 2): $\Sym^m\Omega^1_{U} = \mathcal{F}_{|U} \oplus \pi_* (\widehat{\mathcal{F}}^\perp_{|\widehat{U} \backslash E})$. We set $\mathcal{G} := j_*\pi_* (\widehat{\mathcal{F}}^\perp_{|\widehat{U} \backslash E})$, where $j:U \to X$ is the canonical inclusion. Because all sheaves are locally free over $U$ that is an open subset whose complement has codimension at least 2, we deduce by taking the direct image by $j$ the sought decomposition: $\Sym^{[m]}\Omega^{[1]}_X = \mathcal{F} \oplus \mathcal{G}$ proving $(ii)$. Moreover, $\mathcal{F}$ is parallel over $U$, because $\widehat{\mathcal{F}}$ is parallel over $\widehat{U}$, and $\pi$ being a biholomorphic isometry over $X_\reg$, we have $\mathrm{pr}_{\mathcal{F}} (\overline{\sharp \alpha}^{\Sym^m}_{| \mathcal{F}}) = 0$. This shows $(i)$. If we consider any proper saturated subsheaf $\mathcal{F}'$ of $\mathcal{F}$ with vanishing slope, the same arguments ensure that we have a decomposition into orthogonal subbundles over $\widehat{U} \backslash E$: $\widehat{\mathcal{F}} = \widehat{\mathcal{F}'} \oplus \widehat{\mathcal{F}'}^\perp$, and a decomposition of $\mathcal{F}$. By repeating arguments, we can show that $\mathcal{F}$ is polystable.
\end{proof}

\begin{rema}\label{Bochner-Kobayashi property}
    The previous proposition implies that global sections of $\Sym^{[m]}\Omega^{[1]}_X$ are parallel over $X_\reg$. Indeed, if $\xi \in H^0(X, \Sym^{[m]}\Omega^{[1]}_X)$ is a non-zero global section, then it yields an inclusion of sheaves $\mathcal{O}_X \hookrightarrow \Sym^{[m]}\Omega^{[1]}_X$. As an application of $(i)$ of Proposition~\ref{Polystability of vanishing slopes}, we get that $\xi$ is parallel with respect to the Chern connection of $\Sym^{[m]}\Omega^{[1]}_X$ induced by $\omega$ over $X_\reg$, and it satisfies $\langle \overline{\sharp \alpha}^{\Sym^m} \xi, \xi \rangle = 0$. This generalizes to the singular setting \cite[Theorem~1]{Kob80}. As suggested in this article, we can adapt the arguments to derive a similar statement when $K_X$ is semipositive: the subsheaves $\mathcal{F}$ of $\Sym^{[m]} \mathcal{T}_X$ with vanishing slopes are parallel over $X_\reg$, they satisfy the property $\mathrm{pr}_\mathcal{F}(\sharp \alpha^{\Sym^m}_{|\mathcal{F}}) = 0$, there exists an orthogonal subsheaf that splits $\Sym^{[m]} \mathcal{T}_X$, and they are polystable with respect to $\left[ \omega \right]$. 
\end{rema}

\section{Holonomy decomposition}

Let $p \in X_\reg$ be a fixed point, and let us denote by $H_X := \Hol_p(X_\reg, \omega_\reg)$ the holonomy group asscoiated to the metric $\omega$ on $X_\reg$ and by $H_X^\circ$ the restricted holonomy group. Then, there exists a decomposition of the representation induced by $H_X$ into irreducible pieces:
$$
T_p X = E_{0} \oplus \bigoplus_{i \in I} E_{i},
$$
where $E_{0}$ is the subspace of tangent vectors fixed by $H_X$. A second decomposition is induced by the action of $H_X^\circ$:
$$
T_p X = F_{0} \oplus \bigoplus_{j \in J} F_{j},
$$
where $F_{0}$ is the subspace of tangent vectors fixed by $H_X^\circ = \prod_{j \in J} G_j^\circ$. The factor $G_j^\circ$ is a compact subgroup of $U(n_j)$, and $n_j$ is an integer. 

\subsection{Behaviour of metrics and holonomy under quasi-\'{e}tale covers}
\label{Identification quasi-et holonomy}
One can see how these decompositions behave under quasi-\'etale covers similarly to \cite[Prop.~3.5]{GGK19}. Let $\gamma : X' \to X$ be a finite quasi-\'{e}tale cover. The divisor $\gamma^*H_0$ is ample, represented by $\gamma^*\theta$ and the line bundle associated to $mK_{X'} = \gamma^*mK_X$ can be endowed with the smooth metric $h_{K_X'}^m := \gamma^*h_{K_X}^m$. It follows that the curvature of this metric is $\gamma^*\alpha$, the adapted volume satisfies $\mu_{h_{K_{X'}}^m} = \gamma^*\mu_{h_{K_{X}}^m}$, so that we get by pulling pack the Monge--Amp\`{e}re equation satisfied by $\omega$ over $X_\reg$, we have:
$$
(\gamma^*\omega)^n = c \mu_{h_{K_{X'}}^m}.
$$
This equality of current over $\gamma^{-1}(X_\reg)$ can be extended to $X'$, since none of these currents charge pluripolar set by classical Bedford--Taylor theory\footnote{The fact that $\mu_{h_{K_{X'}}^m}$ is non-pluripolar comes from its definition as a volume form on $X_\reg$, and $\gamma^*\omega^n$ does not charge pluripolar set since it is the Monge--Amp\`{e}re measure associated to the bounded psh function $\gamma^*\Phi$.}.
Moreover, we have:
$$
\mu_{h_{K_{X'}}^m}(X') = \mu_{h_{K_{X'}}^m}(\gamma^{-1}(X_\reg)) = (\deg \gamma) \mu_{h_{K_{X}}^m}(X)
$$
so that $\frac{\left[ \gamma^*H_0 \right]^n}{\mu_{h_{K_{X'}}^m}(X')} = \frac{\left[ H_0 \right]^n}{\mu_{h_{K_{X}}^m}(X)} = c$.
Hence, the uniqueness part of \cite[Thm.~6.3]{EGZ09} ensures that $\gamma^*\omega$ coincide with the unique metric $\omega'$, satisfying $\Ric \omega' = \gamma^*\alpha$ and belonging to the cohomology class of $\gamma^*H_0$.

Let us fix $q \in \gamma^{-1}(p)$, and let us denote by $H_{X'} := \Hol_q(X'_\reg, \omega'_\reg)$ and $H_{X'}^\circ$ the associated restricted holonomy group. Then, $\mathrm{d}\gamma_q : (T_qX', h'_q) \to (T_pX, h_p)$ is an isometry. Under this identification of hermitian spaces, we have the following proposition.

\begin{prop}\textup{\cite[Cor.~4.8]{GGK19}}\label{Quasi-et invariance holonomy decomposition}
	With notation as above, we have $H_{X'}^\circ = H_X^\circ$ and $H_{X'} \subseteq H_X$ is a finite index subgroup.    
\end{prop}

\begin{proof}
	The fact that $H_{X'}^\circ = H_X^\circ$ and $H_{X'} \subseteq H_X$ is proved in \cite[Cor.~4.8]{GGK19}. It remains to show that $H_{X'}$ is a finite index subgroup of $H_X$. This comes from the fact that there is a natural surjective morphism between quotient sets $\pi_1(X_\reg)/\pi_1(X'_\reg) \to H_X/H_{X'}$. Because $\gamma$ is finite, so is the first set, hence $H_X/H_{X'}$ is finite.
\end{proof}

\subsection{Reduction to the holonomy cover}
We recall that we have from the decomposition in irreducible subrepresentations of $H_X^\circ$:
$$
T_p X = F_{0} \oplus \bigoplus_{j \in J} F_{j},
$$

Since $H_X^\circ$ is normal in $H_X$, we observe that:
\begin{itemize}
    \item the group $H_X$ stabilizes $F_0$,
    \item for any $g \in H_X$ and any $j \in J$, the subspace $g(F_j)$ remains an irreducible $H_X^\circ$-representation. Hence, there exists $k \in J$ such that $g(F_j) = F_k$. This yields an action $H_X/H_X^\circ \rightarrow \mathfrak{S}_J$.
\end{itemize}
The kernel of the composition $\pi_1(X_\reg) \longrightarrow H_X/H_X^\circ \longrightarrow \mathfrak{S}_J$ gives us a quasi-\'{e}tale cover $\gamma : X^{\mathrm{wh}} \to X$ called the \emph{weak holonomy cover}. We denote by $\omega_\mathrm{wh}$ the pullback metric. By Proposition~\ref{Quasi-et invariance holonomy decomposition}, the decomposition induced by $H_{X^\mathrm{wh}}^\circ$ on $T_q{X^{\mathrm{wh}}}$ for $q \in \gamma^{-1}(p)$ is such that each irreducible piece $F^\mathrm{wh}_i$ is $H_{X^\mathrm{wh}}$-stable, since the action of $H_{X^\mathrm{wh}}/H_{X^\mathrm{wh}}^\circ$ on $\mathfrak{S}_J$ is trivial. As a remark, one can observe that irreducibility of $H_{X^\mathrm{wh}}^\circ$-representations that are $H_{X^\mathrm{wh}}$-stable implies that they are $H_{X^\mathrm{wh}}$-irreducible: if not there is a non-trivial $H_{X^\mathrm{wh}}$-subrepresentation. Hence, it is $H_{X^\mathrm{wh}}^\circ$-stable, which contradicts the irreducibility of $H_{X^\mathrm{wh}}^\circ$. 
By parallel transport, it induces a $\mathcal{C}^\infty$-decomposition of the tangent bundle $TX^{\mathrm{wh}}_\reg$. Thanks to \cite[Cor.~4.10]{GGK19}, there exists a decomposition of the tangent sheaf of $X$ by reflexive sub-$\mathcal{O}_{X^\mathrm{wh}}$-modules:
$$
\mathcal{T}_{X^\mathrm{wh}} = \mathcal{F}_0 \oplus \bigoplus_{j \in J} \mathcal{F}_j.
$$
One can consider similarly the action of the restricted holonomy on $\Omega^1_{X^\mathrm{wh},q}$ and the decomposition into irreducible pieces at $q$ and extend it to $X^\mathrm{wh}$. Because dual of irreducible finite-dimensional representations are irreducible, this procedure coincide with picking the dual decomposition of sheaves. In other words, if:
$$
\Omega^{[1]}_{X^\mathrm{wh}} = \mathcal{V}_0 \oplus \bigoplus_{k \in K} \mathcal{V}_k
$$
is the bundle decomposition arising from pointwise decomposition into irreducible pieces of $H_{X^\mathrm{wh}}^\circ$-representation, then there is a bijection $\sigma : J \simeq K$ and isomorphisms of $H_X^\circ$-representations $\mathcal{V}_{\sigma(j)} \simeq \mathcal{F}^\vee_j$, and $\mathcal{V}_0 \simeq \mathcal{F}^\vee_0$.\\

By definition of parallel transport and the fact that the Levi--Civita connection is torsion-free, the subsheaves $\mathcal{F}_j$ are involutive. Around a point $q$, there exists an analytic simply connected chart $U$ so that the natural map $\Hol_q(U, \omega_\mathrm{wh}) \to \Hol_q^\circ(X^{\mathrm{wh}}_\reg, \omega_\mathrm{wh})$ is an isomorphism, and we have a splitting of K\"ahler manifolds:
\begin{align}\label{Local decomposition of Kahler manifold}
    (U, \omega_\mathrm{wh}) \simeq (U_0, \omega_0) \times \prod_{j \in J} (U_j, \omega_j),
\end{align}
such that $\mathrm{pr}_j^* \mathcal{T}_{U_j} = \mathcal{F}_j$.
We say that \emph{$\mathcal{F}_j$ is of symmetric type} if for any $q$, and any analytic chart $U$, the associated K\"{a}hler manifold $(U_j, \omega_j)$ in the decomposition is locally symmetric\footnote{This property can be rephrased in $\nabla_{\End(T_X^\mathrm{wh})}\textbf{i}\Theta(\mathcal{F}_j, \omega_{\mathrm{wh} |\mathcal{F}_j}) = 0$.}. Thanks to Berger--Simons classification, if $\mathcal{F}_j$ is not of symmetric type, we have $G_j^\circ = \U(n_j), \SU(n_j)$ or $\Sp(\frac{n_j}{2})$.\\

\begin{theo}[= Theorem~\ref{Bound of symmetric tensors}]
    With notation as above, we have:
    $$
    h^0(X, \Sym^{[m]}\Omega^{[1]}_X) \leq \binom{m + \mathrm{rk}(\mathcal{F}_0) - 1}{m}.
    $$
\end{theo}

\begin{proof}

We argue similarly to \cite{Kob80}. For convenience, we will write $J = \left\lbrace 1, \cdots, r \right\rbrace$. The space of holomorphic symmetric differentials satisfies:
\begin{align*}
    H^0(X^\mathrm{wh}, \Sym^{[m]} \Omega^{[1]}_{X^\mathrm{wh}}) & = H^0(X^\mathrm{wh}_\reg, \Sym^m \Omega^{1}_{X^\mathrm{wh}}) \\
    & = \bigoplus_{m_0 + \cdots + m_j = m} \bigotimes_{j=0}^r H^0(X^\mathrm{wh}_\reg, \Sym^{m_j} \mathcal{V}_j).
\end{align*}

We show that, for $1 \leq j \leq r$ and any $m$, $H^0(X^\mathrm{wh}, \Sym^{[m]} \mathcal{V}_j) = 0$. We investigate two cases whether $\mathcal{V}^\vee_j$ is of symmetric type or not.\\

\begin{itemize}
    \item If $\mathcal{V}^\vee_j$ is of symmetric type, then $\omega_j$ is a K\"ahler--Einstein metric. Indeed, let us denote by $\alpha_j := \Ric \omega_j$. The locally symmetric condition implies that $\sharp \alpha_j$ is parallel, so that it commutes with parallel transport. Hence, by Schur's lemma on irreducible representations, we obtain that $\sharp \alpha_j = \lambda \id_{\End T_{U_j}}$. The function $\lambda$ is constant thanks to commutation of $\sharp \alpha_j$ with parallel transport. By semi-positivity of $\alpha_j$, the constant $\lambda$ is non-negative. If it is zero, $\mathcal{V}_j^\vee$ (hence $\mathcal{V}_j$) would be a flat bundle, and the holonomy would be trivial, that implies $j = 0$. Hence, $\lambda > 0$. 
    From this, there is no non-zero symmetric differential $\xi \in H^0(X^{\mathrm{wh}}, \Sym^{[m]} \mathcal{V}_j)$ such that $\langle \overline{\sharp \alpha}^{\Sym^m} \xi, \xi \rangle = 0$ by Remark~\ref{Bochner-Kobayashi property}.\\

    \item If $\mathcal{V}^\vee_j$ is not of symmetric type, then, $G_j^\circ = \U(n_j), \SU(n_j)$ or $\Sp(n_j/2)$, and the representation of $G_j^\circ$ in $\mathcal{V}_{j, q}$ is irreducible. Hence, the induced representation on $\Sym^m \mathcal{V}_{j,q}$ remains irreducible. By contradiction, if there is a non-zero symmetric differential $\xi \in H^0(X^{\mathrm{wh}}, \Sym^{[m]} \mathcal{V}_j)$, then $\xi_q$ is fixed by $G_j^\circ$ by the holonomy principle, which contradicts the irreducibility of the representation.
\end{itemize}

It follows that the inclusion of spaces 
$$
H^0(X^\mathrm{wh}, \Sym^{[m]} \mathcal{V}_0) \hookrightarrow H^0(X^\mathrm{wh}, \Sym^{[m]} \Omega^{[1]}_{X^\mathrm{wh}})
$$ is an isomorphism. But, the holonomy principle states that evaluation at $q$ and parallel transport are mutually inverse isomorphisms between invariant tensors at $p$ and parallel sections.
Hence, this induces that the following composition of applications is injective:
$$
    \xymatrix{(\Sym^m \Omega^1_{X^\mathrm{wh},q})^{H_{X^\mathrm{wh}}} \ar[d]^{\mathrm{par. tr.}} \ar@{^{(}.>}[r] & \Sym^m \mathcal{V}_{0,q} \\ 
    H^0(X^\mathrm{wh}, \Sym^{[m]} \Omega^{[1]}_{X^\mathrm{wh}}) \ar[r]^{\sim}& H^0(X^\mathrm{wh}, \Sym^{[m]} \mathcal{V}_0) \ar@{^{(}->}[u]^{\mathrm{eval.} p}  }
$$
Then, we have:
$$
h^0(X^\mathrm{wh}, \Sym^{[m]} \Omega^{[1]}_{X^\mathrm{wh}}) \leq \binom{m + \mathrm{rk}(\mathcal{F}_0) - 1}{m}.
$$
Since $\gamma$ is quasi-\'{e}tale, the pullback of sections by $\gamma$ yields an injective morphism:
$$H^0(X_\reg, \Sym^m\Omega^1_{X_\reg}) \to H^0(X^{\mathrm{wh}}_\reg, \Sym^m\Omega^1_{X^\mathrm{wh}_\reg}).$$
The sheaves $\Sym^{[m]} \Omega^{[1]}_{X}$ and $\Sym^{[m]} \Omega^{[1]}_{X^\mathrm{wh}}$ being reflexive, the inequality:
$$h^0(X, \Sym^{[m]}\Omega^{[1]}_X) \leq h^0(X^\mathrm{wh}, \Sym^{[m]}\Omega^{[1]}_{X^\mathrm{wh}})$$ holds.
\end{proof}

\begin{rema}
    The arguments involved do not use the structure of the varieties but the vanishing properties given by Remark~\ref{Bochner-Kobayashi property}. Hence, they can be adapted in order to prove a similar statement when $K_X$ is semi-positive:
    $$h^0(X, \Sym^{[m]}\mathcal{T}_X) \leq \binom{m + \mathrm{rk}(\mathcal{F}_0) - 1}{m}.$$
\end{rema}

\section{Analysis of the bound}

The goal of this section is to show the following result.

\begin{theo}[=Theorem~\ref{Identification euclidean factor augmented irregularity}]
	Let $X$ be a klt projective variety such that $-K_X$ is semipositive. Then $\mathrm{rk}(\mathcal{F}_0) = \widetilde{q}(X)$.
\end{theo}

First, the quantity $\mathrm{rk}(\mathcal{F}_0)$ is invariant by quasi-\'etale covers, since we can compare holonomy decompositions of the tangent sheaves on quasi-\'{e}tale covers by \cite[Remark~6.5]{GGK19}. Let us consider $f : X' \to X$ a quasi-\'etale cover, and $Z$ the normalization of the fiber product of $X^\mathrm{wh} \times X'$. By applying the previous bound to $m = 1$ to $Z$, we have $q(Z) \leq \mathrm{rk}( \mathcal{F}_0)$. By injectivity lemma \cite[Lemma~4.1.14]{Laz04} we deduce $q(X') \leq q(Z)$, hence by passing to the sup in $X'$, we have $\widetilde{q}(X) \leq \mathrm{rk}(\mathcal{F}_0)$.

\begin{rema}
    We can adapt this argument in order to show that the quantity:
    $$
    \widetilde{t}(X) := \sup \left\lbrace h^0(X', \mathcal{T}_{X'}) \,\,|\,\, X' \to X \text{ quasi-\'{e}tale } \right\rbrace
    $$
    is finite, bounded by $\mathrm{rk}(\mathcal{F}_0)$ when $c_1(X) \leq 0$.
\end{rema}

Up to passing to a finite quasi-\'{e}tale cover, we can assume that $X$ is maximally quasi-\'{e}tale, satisfying $q(X) = \widetilde{q}(X)$, and that the decomposition of irreducible pieces of $H_X^\circ$ at a point can be extended into a decomposition of the tangent sheaf. Finally, we can assume that there exists a MRC-fibration $f : X \to Y$ thanks to \cite[Thm.~1.1]{MW21}, where $Y$ is a projective klt variety with numerically trivial canonical class and $f$ is a locally constant holomorphic MRC-fibration. Up to passing to higher quasi-\'{e}tale covers, we can assume thanks to \cite{Dru18, GGK19, HP19} that a Beauville--Bogomolov decomposition holds, that means that we have:
$$
Y = A \times \prod_{i \in I} Y_i \times \prod_{j \in J} Z_j,
$$
where $A$ is an abelian variety, $Y_i$ are irreducible Calabi-Yau (ICY) varieties and $Z_j$ are irreducible holomorphic symplectic (IHS) varieties.
Hence, the universal property of the Albanese morphism ensures that the composition with projection on the abelian factor $a : X \to Y \to A$ is exactly the Albanese morphism of $X$. Hence, $a$ is locally trivial, and its fiber is $F \times \prod_{i \in I} Y_i \times \prod_{j \in J} Z_j$ where $F$ is the general fiber of the MRC-fibration that is rationally connected.\\

Let us show that $h^0(X, \Omega^{[1]}_{X/Y}) = 0$. To do this, we denote by $X_a$ a fiber of $a$. Hence, we have $h^0(X_a,  \Omega^{[1]}_{X_a}) = 0$.
Indeed, we have:
$$
h^0(X_a, \Omega^{[1]}_{X_a}) = h^0(F, \Omega^{[1]}_{F}) + \sum_{i \in I} h^0(Y_i, \Omega^{[1]}_{Y_i}) + \sum_{j \in J} h^0(Z_j, \Omega^{[1]}_{Z_j}).
$$
From the definition of ICY and IHS, we have for any $i,j$, $h^0(Y_i, \Omega^{[1]}_{Y_i}) = h^0(Z_j, \Omega^{[1]}_{Z_j}) = 0$, and by \cite[Thm.~5.1]{GKKP11}, we have $h^0(F, \Omega^{[1]}_F) = 0$.\\

\noindent
\textbf{Step 1 : Construction of a flat subsheaf of $\Omega^{[1]}_{X/Y}$}\\
We fix $p \in X_\reg$. The holonomy decomposition of the cotangent bundle at this point is:
$$
    \Omega^{[1]}_{X, p} = V_0 \oplus \bigoplus_{i \in I} V_i.
$$
Let us show that $(a^*\Omega^1_A)_p \subseteq V_0$, and more generally, $a^*\Omega^1_A \subseteq \mathcal{V}_0$. By contradiction, we assume that that there exists $j \in J$ such that the orthogonal projection $ \mathrm{pr}_j : a^*\Omega^1_A \to \mathcal{V}_j$ is not zero. Hence, because $\Omega^1_A \simeq \mathcal{O}_A^{\oplus q}$, we get that there exists a non-zero global section of $\mathcal{V}_j$. But the proof of Theorem~\ref{Bound of symmetric tensors} establishes that the subsheaves $\mathcal{V}_j$ do not have non-zero sections.

We conclude that $a^*\Omega^1_A \subseteq \mathcal{V}_0$, and this bundle is parallel by Proposition~\ref{Polystability of vanishing slopes}. Hence, the orthogonal bundle $\mathcal{V}'$ provides a subbundle of $\mathcal{V}_0$ defined on $X_\reg$. It is included in $\Omega^{[1]}_{X/A}$.\\

\noindent
\textbf{Step 2 : $\mathcal{V}'$ has rank 0}\\
Since the bundle $\mathcal{V}'$ is included in $\mathcal{V}_0$, it is fixed by the restricted holonomy, hence it is flat over $X_\reg$. Then, this bundle is given by a representation $\pi_1(X_\reg) \to \GL(r', \C)$
where $r' := \mathrm{rk}(\mathcal{V}')$. Because $X$ is maximally quasi-\'{e}tale, this representation factors through $\pi_1(X)$ by \cite[Section~8.1]{GKP16}. By restriction to a fiber of the Albanese morphism, this provides a morphism $\rho : \pi_1(X_a) \to \GL(r', \C)$. The image of this representation is finite, since $X_a = F \times Y'$ where $F$ is rationally connected and $Y'$ is a product of IHS and ICY varieties. Hence, we have $\pi_1(F) = 1$, and the image of any representation of $\pi_1(Y')$ is finite (see \cite[Rem.~5.3]{Gue25}). There exists a finite \'{e}tale cover $f : \widetilde{X_a} \to X_a$ such that $f^*\mathcal{V}' = \mathcal{O}_{\widetilde{X_a}}^{\oplus r'}$. Such a bundle is a subbundle of $f^{[*]}\Omega^{[1]}_{X_a} = \Omega^{[1]}_{\widetilde{X_a}}$. The goal is to show that $h^0(\widetilde{X_a}, \Omega^{[1]}_{\widetilde{X_a}}) = 0$, so that it implies that $r' = 0$, hence $\mathrm{rk}(\mathcal{V}_0) = \mathrm{rk}(\Omega^1_A) = \widetilde{q}(X)$.\\

Since $X_a \simeq F \times Y'$ and $F$ is simply connected, an \'{e}tale finite cover of $X_a$ is of the form $F \times \widetilde{Y'}$ where $\widetilde{Y'} \to Y'$ is an \'{e}tale finite cover. Hence, $\widetilde{X_a} \simeq F \times \widetilde{Y'}$. Because $Y'$ has vanishing augmented irregularity, we deduce that:
$$
h^0(\widetilde{X_a}, \Omega^{[1]}_{\widetilde{X_a}}) = h^0(F, \Omega^{[1]}_F) + h^0(\widetilde{Y'}, \Omega^{[1]}_{\widetilde{Y'}}) = 0. 
$$\qed

\begin{rema}\label{Coincidence of decompositions albanese holonomy}

    A consequence of Step 2 is the decomposition of the cotangent sheaf $\Omega^{[1]}_X = a^* \Omega^{[1]}_A \oplus \Omega^{[1]}_{X/A}$ in a compatible way with holonomy:
    $$
    a^*\Omega^1_A = \mathcal{V}_0 \;\;\text{ and }\;\; \Omega^{[1]}_{X/A} = \bigoplus_{j \in J} \mathcal{V}_j.
    $$
    Hence, since we have $\mathcal{V}_j = \mathcal{F}_j^\vee$ up to permutations, we deduce:
    $$
     a^*\mathcal{T}_A = \mathcal{F}_0 \;\;\text{ and }\;\; \mathcal{T}_{X/A} = \bigoplus_{j \in J} \mathcal{F}_j.
    $$
    It follows that the decomposition $\mathcal{T}_X = a^*\mathcal{T}_A \oplus \mathcal{T}_{X/A}$ holds and $\mathcal{T}_{X/A}$ is parallel with respect to the Levi--Civita connection on $X_\reg$.
    
    Moreover, the Albanese morphism $X \to A$ is locally constant, hence, there exists around $a(p)$ a neighborhood $V$ and an isomorphism $\tau : a^{-1}(V) \simeq V \times X_a$. This decomposition is compatible with the decomposition induced by holonomy:
    $$
    (U,\omega) \simeq (U_0, \omega_0) \times \prod_{j \in J} (U_j, \omega_j)
    $$
    in the following sense: there are two inclusions $\tau_1 : U_0 \hookrightarrow V$ and $\tau_2 : \prod_{j \in J} U_j \hookrightarrow X_a$ such that $\tau_{|U} = (\tau_1, \tau_2)$.
\end{rema}

\section{Compactness of the holonomy}

The aim of this section is to show that $\Hol_p(X_\reg, \omega_{X_\reg})$ is compact.
First, we endow $X_a$ with the metric $\omega_a := i^*\omega$ where $i : X_a \to X$ is the canonical inclusion of the fiber passing through $p$. The metric $\omega_a$ is the restriction of the metric $\omega$ to $X_a$. By restriction of the decomposition $\mathcal{T}_X = a^*\mathcal{T}_A \oplus \mathcal{T}_{X/A}$ explained in Remark~\ref{Coincidence of decompositions albanese holonomy} to $X_a$, we have a decomposition into hermitian subbundles:
$$
    \mathcal{T}_{X|X_a} = (a^*\mathcal{T}_{A})_{|X_a} \oplus \mathcal{T}_{X_a}.
$$
For a loop $\gamma$ in $X_a$ based at $p$, we denote by $P^X_\gamma$ (resp. $P^{X_a}_\gamma$) the parallel transport along $\gamma$ computed in $X$ (resp. $X_a$). We have $P^X_\gamma(T_{X_a, p}) = T_{X_a, p}$ and $P^X((a^*\mathcal{T}_{A})_p) = (a^*\mathcal{T}_{A})_p$. Moreover, the vectors of $(a^*\mathcal{T}_{A})_p$ being invariant under holonomy in $X$ by Remark~\ref{Coincidence of decompositions albanese holonomy}, we get that the parallel transport on $(a^*\mathcal{T}_{A})_p$ is the identity. We deduce that there exists a natural inclusion of groups:
$$
    P^{X_a}_\gamma \in H_{X_a} \longmapsto
    P^X_\gamma = \begin{bmatrix}
        \id_{T_{A, a(p)}} & 0 \\
        0 & P^{X_a}_\gamma
    \end{bmatrix} 
    \in H_X.
$$
Under this injection, we have $H^\circ_{X_a} \subseteq H^\circ_X$ since a path homotopic to a constant in $X_a$ is homotopic to a constant in $X$. The following lemma enlightens the relationship between these holonomy groups. 

\begin{lemm}\label{Lemma H_Xa}
	With notation as above:\\
	$(i)$ $H_X^\circ \simeq H_{X_a}^\circ$,\\
	$(ii)$ the group $H_{X_a}$ is compact.	
\end{lemm}

\begin{proof}
	For $(i)$, there exists a simply connected neighborhood $U$ of $p$ such that:
	$$
	(U, \omega) \simeq (U_0, \omega_0) \times \prod_j (U_j, \omega_j),
	$$
	and $\mathrm{pr}_j^*\mathcal{T}_{U_j} = \mathcal{F}_j$. Hence, thanks to 	Remark~\ref{Coincidence of decompositions albanese holonomy}, 	we get that $\prod_j (U_j, \omega_j)$ is an analytic chart of $(X_a, \omega_a)$.	Hence, we have $H_{\prod U_j} \simeq H_{X_a}^\circ$. But we have:
	$$H_{\prod_j U_j} = \prod_j H_{U_j} = \prod_j G_j^\circ = H_X^\circ.$$
	
	For $(ii)$, it suffices to show that $H_{X_a}/H_{X_a}^\circ$ is finite.  We have the following diagram whose each square is commutative:
	$$
	\xymatrix{
	\pi_1({X_a}_{\reg}) \ar@{^(->}[r] \ar@{->>}[d] \ar@/_3pc/@{->>}[dd] & \pi_1(X_\reg) \ar@{->>}[d] \\
	\pi_1(X_a) \ar@{^(->}[r] & \pi_1(X) \ar@{->>}[d] \\
	H_{X_a}/H_{X_a}^\circ \ar@{^(->}[r] & H_X/H_X^\circ
	}.
	$$
	
	By the Serre exact sequence associated to a fibration and the fact that $\pi_2(A) = 0$, horizontal arrows are injective, and the vertical arrows are surjective. The existence of the map $\pi_1(X) \to H_X/H_X^\circ$ is given by the fact that any linear representation of $\pi_1(X_\reg)$ factors through $ \pi_1(X)$ because $X$ is assumed maximally quasi-étale. Now, any representation of $\pi_1(X_a)$ has finite image, so that the image of the composition $\pi_1(X_a) \to \pi_1(X) \to H_X/H_X^\circ$ is finite. We deduce that the precomposition with $\pi_1({X_a}_\reg) \to \pi_1(X)$ remains of finite image. Because all squares are commutative, and the map $\pi_1({X_a}_\reg) \to H_{X_a}/H_{X_a}^\circ$ is surjective, we deduce that the map $H_{X_a}/H_{X_a}^\circ \to H_X/H_X^\circ$ has finite image. But this map is injective, so that $H_{X_a}/H_{X_a}^\circ$ is finite.
\end{proof}

We now show the compactness of the holonomy group.

\begin{theo}[= Theorem~\ref{Compactness of holonomy}]
	With notation as above, the group $H_X$ is compact.
\end{theo}

From Lemma~\ref{Lemma H_Xa}, we have the inclusions: $H_{X_a}^\circ = H_X^\circ \subseteq H_{X_a} \subseteq H_X$. Because $\pi_1({X_a}_\reg)$ is normal in $\pi_1(X_\reg)$, so is $H_{X_a}$ in $H_X$. Hence, there is a natural surjective map $\rho : \pi_1(A) \to H_X/H_{X_a}$.

$$
\xymatrix{
	1 \ar[r] & \pi_1({X_a}_\reg) \ar[d] \ar[r] & \pi_1(X_\reg) \ar[r] \ar[d] & \pi_1(A) \ar@{.>}[d]_\rho \ar[r] & 1 \\
	1 \ar[r] & H_{X_a}/H_{X}^\circ \ar[r] & H_X/H_X^\circ \ar[r] & H_X/H_{X_a} \ar[r] & 1
	}.
$$

The goal is to show that its image is finite. Let us denote by: 
$$
\Aut(X_a,\omega_a) = \lbrace f \in \Aut(X_a) \;\;|\;\; f^*\omega_a = \omega_a \rbrace
$$
the group of automorphisms of $X_a$ preserving the metric $\omega_a$. 

\begin{lemm}	\label{Compactness of aut}
	The group $\Aut(X_a, \omega_a)$ is compact.
\end{lemm}

\begin{proof}
	We endow $X_a$ with the analytic topology, so that it is a completely metrizable space, and we see $\Aut(X_a, \omega_a)$ as a subset of $\mathcal{C}^0(X_a, X_a)$ endowed with the topology of uniform convergence.
	Let $(f_n)_n$ be a sequence of elements of $\Aut(X_a, \omega_a)$. Hence, they are isometries of $({X_a}_\reg, \omega_a)$. One can consider the metric completion $(X_a', d')$ of this space which is compact by \cite[Cor.~9.1]{GPSS23}, and each $f_n$ extends to an isometry of $(X_a', d')$. By Arzela-Ascoli theorem, we can extract a subsequence of $(f_n)$ converging uniformly to a map $f_\infty : X_a' \to X_a'$. Hence, it is holomorphic and smooth over ${X_a}_\reg$. Moreover, this map is an isometry. We recall that the metric $\omega_a$ belongs to the ample class $[H_0]_{|X_a}$, and for $n \in \Z_{>0}$, the automorphism $f_n$ preserves $\omega_a$, hence it preserves the class $\left[ H_0 \right]_{|X_a}$. Since klt singularities are rational, we have $h^1(X_a, \mathcal{O}_{X_a}) = h^0(X_a, \Omega^{[1]}_{X_a}) = 0$, we deduce that there is an isomorphism of line bundles $f_n^*{H_0}_{|X_a} \simeq {H_0}_{|X_a}$. Then, the precomposition by $f_n$ defines an automorphism of $H^0(X_a, k{H_0}_{|X_a})$ for any $k$. This isomorphism depends on the isomorphism $f_n^*{H_0}_{|X_a} \simeq {H_0}_{|X_a}$, but two such isomorphisms differ by a multiplicative constant. Hence, $f_n$ induces a unique automorphism of $\mathbb{P}(H^0(X_a, k{H_0}_{|X_a})^\vee)$ for any $k$. We fix $k$ such that the map $\varphi : X_a \to \mathbb{P}(H^0(X_a, k{H_0}_{|X_a})^\vee)$ defines an embedding. There is a unique morphism $f_n^*$ such that the following diagram is commutative:
	$$
	\xymatrix{
	X_a \ar[r]^{f_n} \ar[d]_{\varphi} & X_a \ar[d]^\varphi \\
	\mathbb{P}(H^0(X_a, k{H_0}_{|X_a})^\vee) \ar[r]_{f_n^*} & \mathbb{P}(H^0(X_a, k{H_0}_{|X_a})^\vee).
	}
	$$
	One can extract a subsequence of $(f_n^*)$ that converges pointwise\footnote{To do this, we endow the line bundle $H_0$ with a smooth hermitian metric $h$, so that $H^0(X_a, k{H_0}_{|X_a})$ can be naturally endowed with a hermitian norm $|\sigma|^2 := \int_X |\sigma|_h^2 \omega_a^d$ where $d := \dim X_a$. Hence, because each $f_n$ is $\omega_a$-invariant, the precomposition by $f_n$ is a unitary endomorphism of $H^0(X_a, k{H_0}_{|X_a})$. It turns out that we can extract a subsequence that converges in $\U(H^0(X_a, k{H_0}_{|X_a}), |\cdot|)$ to an endomorphism $f$. Since the reduction in $\PGL(d, \C)$ is continuous, the corresponding  extracted sequence of $(f_n^*)$ converges pointwise to $f^*$.} to an automorphism $f^*$ of $\mathbb{P}(H^0(X_a, k{H_0}_{|X_a})^\vee)$. It follows that this automorphism extends $f_\infty$ to $\mathbb{P}(H^0(X_a, k{H_0}_{|X_a})^\vee)$ and it preserves $X_{a \reg}$, hence $X_a$. Then, its restriction to $X_a$ is an automorphism of $X_a$ extending $f_\infty$. The convergence of the subsequence of $(f_n)$ to $f_\infty$ is uniform, since it is uniform on $X_\reg$ that is a dense open subset.\\
	
	The map $f_\infty$ is an automorphism of $X_a$. Because it is extended by an automorphism $f^*$ of a projective space, we deduce that $f_\infty^*{H_0}_{|X_a} \simeq {H_0}_{|X_a}$. One can assume that the reference smooth metric $\theta$ representative of $\left[ H_0 \right]$ is the restriction of the Fubini--Study metric through the embedding $X \hookrightarrow \mathbb{P}(H^0(X, kH_0)^\vee)$. Hence, restricting this metric to the subvariety $X_a$ guarantees the existence of a smooth function $\Psi$ on $X_a$ such that $f_\infty^*(i^*\theta) = i^*\theta + \ddc \Psi$. Since $f_\infty^*\omega_a = \omega_a$ over $X_{a \reg}$, it follows the following equality of Monge-Amp\`{e}re measures over $X_{a\reg}$:
	$$
	(i^*\theta + \ddc (\Psi + f_\infty^*i^*\Phi))^d = (i^*\theta + \ddc i^*\Phi)^d.
	$$
	This equality is true over the whole $X_a$ since potentials are bounded, and Bedford--Taylor theory guarantees that these measures do not charge analytic sets. By uniqueness of the solution, we have $\Psi + f_\infty^*i^*\Phi = i^*\Phi$, over $X_a$, hence, $f_\infty^*\omega_a = \omega_a$.
\end{proof}

\begin{rema}
	We deduce that $\Aut(X_a, \omega_a)^\circ$ is compact and the quotient group $\Aut(X_a, \omega_a)/\Aut(X_a, \omega_a)^\circ$ is finite.
\end{rema}

Now, we can define a map:

$$
\Xi : 
\left\lbrace
\begin{array}{ccc}
	\Aut(X_a, \omega_a) & \longrightarrow & U(T_pX, \omega_p)/H_{X_a} \\
	f & \longmapsto & (\mathrm{d}f_p)^{-1} \circ P_\gamma
\end{array}
\right.
$$
where $P_\gamma$ is the parallel transport along any path $\gamma$ from $p$ to $f(p)$ in $X_a$. This is well-defined, namely independent of the path $\gamma$: if $\gamma'$ is another path, we have: $(\mathrm{d}f_p)^{-1} \circ P_{\gamma'} = (\mathrm{d}f_p)^{-1} \circ P_\gamma \circ P_{\gamma^{-1}\gamma'}$, and $P_{\gamma^{-1}\gamma'}$ belongs to $H_{X_a}$.

\begin{lemm}\label{Continuity of Xi}
	The morphism $\Xi$ is continuous.
\end{lemm}

\begin{proof}
We endow $U(T_pX, \omega_p)/H_{X_a}$ with the quotient topology. It suffices to show that $\Xi$ is continuous\footnote{Indeed, one can observe that $\Xi(f \circ g) = \Xi(f) \Xi(g)$ by using the fact that $f$ is an isometry, hence it commutes with parallel transport. Let us assume that $\Xi$ is continuous at $\id_{X_a}$, and let $(g_n)$ be a sequence converging to $g$. Because the composition by $g^{-1}$ is continuous, the sequence $\Xi(g^{-1} \circ g_n)$ converges to $\mathrm{id}_{T_pX_a}$, and $\Xi(g_n) = \Xi(g) \Xi(g^{-1} \circ g_n)$ converges to $\Xi(g)$.} at $\id_{X_a}$. Let $(f_n)_n$ be a sequence of elements of $\Aut(X_a, \omega_a)$ converging to $\id_{X_a}$. Hence, we have $(f_n(p))_n$ that converges to $p$. We can choose a chart around $p$, and paths $\gamma_n : t \in [0,1] \longmapsto tf_n(p)$.
	First, thanks to Cauchy formula, the pointwise convergence of $(f_n)_n$ implies that the derivatives of $f_n$ converge also to the derivative of $\id_{X_a}$. Hence, the matrices $(\mathrm{d}f_{n,p})^{-1}$ read in the chart converge to the identity matrix. It remains to show that, for any $v \in T_pX$, the parallel transports $P_{\gamma_n}(v)$ converge to $P_{\gamma_\infty}(v)$ where $\gamma_\infty$ is the constant loop at $p$. We study the solutions of the parallel transport equation in charts. For any $n \in \Z_{>0} \cup \lbrace \infty \rbrace$, they are of the form:
	 \begin{align*}
	 	\dot{y_n}(t) & = A_n(t) y_n(t) \\
	 	y_n(0) & = v
	 \end{align*}
where $y_n : \left[0, 1 \right] \to \C^n$ and $A_n : \left[0, 1 \right] \to \mathcal{M}(n, \C)$. The expression of $A_n$ depends on Christoffel symbols associated to the connection, and the coordinates of the path $\gamma_n$ and its derivatives. Namely, we have $(A_n(t))_{ij} = \sum_k \Gamma^k_{ij}(\gamma_n(t))\dot{\gamma_{n,k}}(t)$. Because $\gamma_n$ converges smootly to $\gamma_\infty$, the sequence $A_n$ converges smoothly to $A_\infty = 0$. Hence, we have to show that $P_{\gamma_n}(v) = y_n(1)$ converges to $y_\infty(1)$.
	
	For any $0 \leq x \leq 1$, we have:
	$$
	|y_n(x) - y_\infty(x)|  \leq \int_0^x |A_n| |y_n(t)| \mathrm{d}t,
	$$
	where $|A_n| = \sup_t |A_n(t)|$, and $|A_n(t)|$ is the operator norm. Now, by the Gronwall's inequality, we can show that $y_n$ is bounded, and by Lebesgue's dominated convergence, we get that $(y_n)$ converges pointwise to $y_\infty$, hence the result.
\end{proof}

We introduce now a representation $M : \pi_1(A) \to \Aut(X_a)$ arising from the local constancy of of the MRC-fibration. Indeed, we denote by $f : X \to Y$ the MRC-fibration (up to replace $X$ by a finite quasi-\'{e}tale cover). By \cite[Thm~1.1]{MW21}, there is a morphism $\pi_1(Y) \to \Aut(F)$, and we have $Y = A \times \prod Y_i \times \prod Z_j$. Moreover, the group $\pi_1(A)$ is naturally embedded in $\pi_1(Y)$ and since $X_a = F \times \prod Y_i \times \prod Z_j$ the group $\Aut(F)$ is naturally included in $\Aut(X_a)$. It turns out that we have a natural monodromy morphism $M : \pi_1(A) \to \Aut(X_a)$.

\begin{lemm}\label{Holonomy vs monodromy}
	With notation as above.\\
	$(i)$ The image of $M$ is included in $\Aut(X_a, \omega_a)$,\\
	$(ii)$ The image of the composition $\Xi \circ M$ is $H_X/H_{X_a}$, and $\Xi \circ M = \rho$.
\end{lemm}

\begin{proof}
	\textbf{Step 1 : Construction of flat lifts of paths of $A$}\\
	Let $\eta : \left[ 0,1 \right] \to A$ be a loop based at $a(p)$ in the Albanese variety. We construct a unique smooth lift $\widetilde{\eta} : \left[0,1 \right] \to X_\reg$ such that $\dot{\widetilde{\eta}}(t) \in (a^*\mathcal{T}_A)_{\widetilde{\eta}(t)}$ for any $t$, and $\widetilde{\eta}(0) = p$. The construction is as follows.
	We cover the loop $\eta$ by a finite number of trivializing open subsets $(U_0, \cdots, U_N)$. By this, we mean that each $U_i$ is an open subset of $A$ such that there exists an isomorphism $\tau_i : a^{-1}(U_i) \simeq U_i \times X_a$. By local constancy of the fibration, there exists for any $0 \leq i,j \leq N$ an automorphism $g_{ij} \in \Aut(X_a)$ such that the composition:
	$$
	U_{ij} \times X_a \overset{\tau_j^{-1}}{\longrightarrow} a^{-1}(U_{ij}) \overset{\tau_{i}}{\longrightarrow}  U_{ij} \times X_a
	$$
	is of the form $(x, f) \longmapsto (x, g_{ij}(f))$ (here, we have denoted $U_{ij} = U_i \cap U_j$).
	
	Let $m \in \Z_{> 0}$ be an integer such that for any $0 \leq k \leq m-1$, there exists $U_{i_k}$ such that $\eta(\left[ \frac{k}{m}, \frac{k+1}{m} \right]) \subseteq U_{i_k}$. Hence we construct $\widetilde{\eta}$ by induction on $k$. For $k = 0$, we set:
	$$\widetilde{\eta} : t \in \left[ 0, \frac{1}{m} \right] \to \tau_{i_0}^{-1}(\eta(t), f_0) \in X,
	$$
	where $\tau_{i_0}^{-1}(\eta(0), f_0) = p$. If $\widetilde{\eta} : \left[ 0, \frac{k}{m} \right] \to X$ is constructed, we extend it to $\left[0, \frac{k+1}{m} \right]$ by the following. We denote $\tau_{i_k}^{-1}(\widetilde{\eta}(\frac{k}{m})) = (\eta(\frac{k}{m}), f_{k})$. Hence, we extend $\widetilde{\eta}$ by:
	$$ \widetilde{\eta} : t \in \left[ \frac{k}{m}, \frac{k+1}{m} \right] \longmapsto \tau_{i_k}^{-1}(\eta(t), f_k) \in X.
	$$
	
	It remains to show that $\widetilde{\eta}$ satisfies the sought properties, namely $\dot{\widetilde{\eta}}(t) \in (a^*\mathcal{T}_A)_{\widetilde{\eta}(t)}$. It suffices to observe that trivializations $\tau_i$ send tangent vectors in $a^*\mathcal{T}_A$ in $\mathrm{pr}_1^*\mathcal{T}_{U_i}$ thanks to commutativity of the trivialization with the Albanese morphism and projections. Then, the derivative of the path $(\eta(t), f_k)$ belongs to $\mathrm{pr}_1^*\mathcal{T}_{U_i}$, so $\dot{\widetilde{\eta}}(t)$ belongs to $a^*\mathcal{T}_A$.

\begin{rema}\label{parallel transport of the flat lift vs monodromy}
    One can establish a link between the parallel transport of the flat lift and the monodromy. Let $g$ be an element of the image of $M$. Hence, there exists a smooth loop $\eta$ in $A$ such that $M(\eta) = g$. By Step 1, there exists a flat lift $\widetilde{\eta}$ of $\eta$. Since we have $f_k = g_{i_k i_{k-1}}(f_{k-1}) $, the automorphism $M(\eta)$ corresponds to:
	$$
	M(\eta) = g_{i_0i_{m-1}} \circ \cdot \circ g_{i_1i_0}, 
	$$
	and we have $\widetilde{\eta}(1) = \tau_{i_0}^{-1}(\eta(0), M(\eta)f_0)$.
    We consider the parallel transport along this path, denoted $P_{\widetilde{\eta}}$. Let $v \in T_pX_a$ be a tangent vector of the fiber of $a$ passing through $p$. Then, $P_{\widetilde{\eta}}$ consists of local parallel transports along flat directions. More precisely, we have:
    $$
    P_{\widetilde{\eta}}(v) = \mathrm{d}g_{i_0i_{m-1}} \circ P_{\widetilde{\eta}_{m-1}} \circ \mathrm{d}g_{i_{m-1}i_{m-2}} \circ \cdots \circ \mathrm{d}g_{i_1i_0} \circ P_{\widetilde{\eta}_{0}}(v),
    $$
    where $P_{\widetilde{\eta}_{k}}$ is the parallel transport computed through the trivialization $\tau_{i_k}$. Thanks to Remark~\ref{Coincidence of decompositions albanese holonomy}, the computation of $P_{\widetilde{\eta}_{k-1}}$ is local, and regarding to the local decomposition of K\"ahler manifolds:
    $$
    (U,\omega) \simeq (U_0, \omega_0) \times \prod_{j \in J} (U_j, \omega_j),
    $$
    we are reduced to the case where the path is included in a fiber in the product (namely, it is included in a submanifold of the form $U_0 \times \lbrace \mathrm{pt} \rbrace$), and we compute the parallel transport of some orthogonal vector (namely, this vector belongs to $\mathcal{T}_{\prod_j U_j}$). In this case, the parallel transport is the identity, so that:
    $$
    P_{\widetilde{\eta}}(v) = \mathrm{d}g_{i_0i_{m-1}} \circ \mathrm{d}g_{i_{m-1}i_{m-2}} \circ \cdots \circ \mathrm{d}g_{i_1i_0}(v)  = \mathrm{d}g_p(v).
    $$ 
\end{rema}

\noindent
\textbf{Step 2 : Proof of the lemma}\\
Let us prove $(i)$. From Remark~\ref{parallel transport of the flat lift vs monodromy} and the fact that parallel transport is an isometry of euclidean spaces, we have $\mathrm{d}g_p : (T_pX_a, \omega_{a, p}) \to (T_{g(p)}X_a, \omega_{a, g(p)})$ that is an isometry. This works for any point $p$ in the same fiber $X_a$, hence $g$ is an isometry of $({X_a}_\reg, \omega_{a \reg})$. Hence we have $g^*\omega_a = \omega_a$ over $X_{a \reg}$ hence on the whole $X_a$ by the same arguments as the end of the proof of Lemma~\ref{Compactness of aut} (namely, $g^*\omega_a^d$ is a Monge--Amp\`{e}re measure that does not charge analytic subsets). Hence, $g \in \Aut(X_a, \omega_a)$.\\

Now, we prove $(ii)$. We consider $P_\alpha$ a representative of an element of $H_X/H_{X_a}$, where $\alpha$ is a loop of $X$. Hence, one can construct the flat lift $\widetilde{\eta}$ of the loop $\eta := a \circ \alpha^{-1}$. We show that there exists a unique path $\kappa$ from $p$ to $\widetilde{\eta}(1)$ up to homotopy such that $\alpha^{-1}$ is homotopic to $\kappa^{-1}\widetilde{\eta}$. 
We cover like in Step 1 by trivializing open subsets. We have $\eta (\left[ \frac{k}{m}, \frac{k+1}{m} \right]) \subseteq U_{i_k}$, and we have $\tau_{i_k}(\alpha^{-1}(t)) = (\eta(t), \kappa_k(t))$. for $\frac{k}{m} \leq t \leq \frac{k+1}{m}$. We set:
$$
\kappa : 
\left\lbrace
\begin{array}{cccc}
	\left[ 0, 1 \right] & \longrightarrow & X_a & \\
	t & \longmapsto	& g_{i_0i_1} \circ & \cdots \circ g_{i_{k-1}i_k}\kappa_k(t) \text{ if } \frac{k}{m}\leq t \leq \frac{k+1}{m}
\end{array}
\right.
$$
This is well-defined and continuous, since $\kappa_{k-1}(\frac{k}{m}) = g_{i_{k-1}i_k}\kappa_k(\frac{k}{m})$. It remains to show that $\alpha^{-1}$ is homotopic to $\kappa^{-1}\widetilde{\eta}$. To do so, we make two observations:
\begin{itemize}
	\item The path $\tau_{i_k}^{-1}(\alpha^{-1}_{| \left[\frac{k}{m}, \frac{k+1}{m}\right]})$ is homotopic to the concatenation of the paths $\left( \eta(\frac{k+1}{m}), \kappa_k \right) \left(\eta, \kappa_k(\frac{k}{m})\right)$  by the product structure in the trivialization.\\
	\item On $\left[0, \frac{1}{m} \right]$, for any $0 \leq p \leq k-1$ the following paths in $U_{i_k} \times X_a$
	\scriptsize
	$$\left( \eta\left( \frac{k+ \cdot}{m} \right), g_{i_ki_{k-1}} \circ \cdots \circ \cdots g_{i_{p+1}i_p}\kappa_p \left( \frac{p+1}{m} \right) \right) \left( \eta\left( \frac{k}{m} \right), g_{i_ki_{k-1}} \circ \cdots \circ \cdots g_{i_{p+1}i_p}\kappa_p \left( \frac{p + \cdot}{m} \right) \right),$$
	
	$$
	\left( \eta\left( \frac{k+1}{m} \right), g_{i_ki_{k-1}} \circ \cdots \circ \cdots g_{i_{p+1}i_p}\kappa_p \left( \frac{p + \cdot}{m} \right) \right) \left( \eta\left( \frac{k+ \cdot}{m} \right), g_{i_ki_{k-1}} \circ \cdots \circ \cdots g_{i_{p+1}i_p}\kappa_p \left( \frac{p}{m} \right) \right) 
	$$
	\normalsize
	are homotopic.
\end{itemize}
By induction, we can deduce that $\alpha^{-1}$ is homotopic to $\kappa^{-1}\widetilde{\eta}$. Hence, $\alpha$ is homotopic to $\widetilde{\eta}^{-1}\kappa$. From the equality $P_{\widetilde{\eta}} = \mathrm{d}g_p$ of Remark~\ref{parallel transport of the flat lift vs monodromy}, we deduce that:
	$$\rho(a_*\alpha) = P_\alpha = P_{\widetilde{\eta}^{-1}\kappa} = P_{\widetilde{\eta}}^{-1} \circ P_\kappa = (\mathrm{d}g_p)^{-1} \circ P_\kappa = \Xi(M(a_*\alpha)).$$
\end{proof}

\begin{proof}[Proof of Theorem~\ref{Compactness of holonomy}]
	We consider the closure $G$ of $\mathrm{Im}(M)$ in $\Aut(X_a, \omega_a)$. It is a compact subgroup by Lemma~\ref{Compactness of aut}. Hence, $G$ has a finite number of components. Each connected component is sent in a connected component of $H_X/H_{X_a}$ by $\Xi$ by Lemma~\ref{Continuity of Xi}. The latter group is totally disconnected in $U(n)/H_{X_a}$, so that $\Xi(G)$ is finite. But $\Xi(G)$ contains the image of $\Xi \circ M = \rho$ by Lemma~\ref{Holonomy vs monodromy} which is surjective, showing that $H_X/H_{X_a}$ is finite. We deduce that $H_X$ is a finite extension of a compact group, which is compact.
\end{proof}

As a consequence of Theorem~\ref{Compactness of holonomy}, the group $H_X/H_X^\circ$ is finite and we can construct a holonomy cover of $X$ similarly to \cite[Lemma~7.8]{GGK19}.

\begin{defi}
    A \emph{holonomy cover of $X$} is a finite quasi-\'{e}tale cover $f : X^\hol \to X$ given by the kernel of the map $\pi_1(X_\reg) \to H_X/H_X^\circ$ under the correspondence \cite[Thm.~23]{CGG24}. Under identifications explained in Section~\ref{Identification quasi-et holonomy}, we have:
    $$H_{X^\hol} = H_{X^\hol}^\circ = H_X^\circ.$$
\end{defi}

\section{Irreducible pieces}

In the case where $K_X$ is numerically trivial, the singular Beauville--Bogomolov decomposition shows that each variety can be split into irreducible pieces up to some quasi-\'{e}tale cover. These irreducible pieces are abelian varieties, irreducible Calabi--Yau and irreducible holomorphic symplectic varieties. In our case, the result of \cite{CDP14} suggests that new irreducible pieces appear, that have algebraic properties (namely, they are rationally connected). We show that these pieces are exactly the one being hereditarily rationally connected in the sense of \cite[Def.~4.3]{GLM23}.

\begin{defi}
	We say that a projective variety $X$ is \emph{hereditarily rationally connected (HRC)} if for any quasi-\'{e}tale morphism $X' \to X$, the variety $X'$ is rationally connected (RC). 
\end{defi}

\begin{prop}\label{Link holonomy HRC}
	Let $X$ be a klt projective variety such that $-K_X$ is semipositive. Assume that the restricted holonomy acts irreducibly at a point $p \in X_\reg$. The following assertions are equivalent:\\
	$(i)$ $\Hol_p(X_\reg, \omega_\reg) = \U(n)$ or $(X_\reg, \omega_\reg)$ is locally symmetric and $\widetilde{q}(X) = 0$.\\
	$(ii)$ $X$ is HRC.
\end{prop}

\begin{proof}

We first show $(i) \Longrightarrow (ii)$.
We consider $\gamma : X' \to X$ a finite quasi-\'{e}tale morphism. By \cite{MW21}, there exists a further quasi-\'{e}tale morphism $\gamma' : X'' \to X'$ and a MRC-fibration $f : X'' \to Y$ with $c_1(Y) = 0$, that is locally constant. We denote by $\omega'' := (\gamma \circ \gamma')^*\omega$ and $\alpha'' := (\gamma \circ \gamma')^*\alpha$.

It suffices to show that $X''$ is RC to prove that $X'$ is RC, or in other words, the variety $Y$ is a point. From Theorem~4.7 of the aforementioned article, the MRC-fibration comes from a splitting of tangent sheaves, that yields a splitting on the reflexive cotangent sheaf: $\Omega^{[1]}_{X''} = f^{[*]} \Omega^{[1]}_Y \oplus \Omega^{[1]}_{X''/Y} $. We observe that $f^{[*]} \Omega^{[1]}_Y$ is a direct summand of $\Omega^{[1]}_{X''}$ with vanishing slope. We can apply Proposition~\ref{Polystability of vanishing slopes} to $f^{[*]}\Omega^{[1]}_Y$ to get that it is a polystable subsheaf, parallel with respect to the Chern connection of $\omega''$. In particular, at a point $q \in X''_\reg \cap f^{-1}(Y_\reg)$ over $p$, we have the decomposition:
$(f^*\Omega^1_Y)_q \oplus (f^*\Omega^1_Y)^\perp_q = \Omega^1_{X'', q}$ that stabilizes the restricted holonomy. The action of the holonomy being irreducible  on $\Omega^1_{X'', q}$, we have $(f^*\Omega^1_Y)_q = \Omega^1_{X'', q}$ or $(f^*\Omega^1_Y)_q = 0$.\\

Let us first assume that $\Hol_p(X_\reg, \omega_\reg) = \U(n)$. Hence, by Proposition~, we have $\Hol_q(X''_\reg, \omega''_\reg) = \U(n)$. By contradiction, if $(f^*\Omega^1_Y)_q = \Omega^1_{X'', q}$, then, $f$ is an isomorphism. In particular, $K_{X''}$ is numerically trivial. We deduce that:
$$
	\int_{X''} \alpha'' \wedge \theta^{n-1} = \left[ -K_{X''} \right] \cdot \left[ H \right]^{n-1} = 0,
$$
hence, by semi-positivity of $\alpha''$, we deduce that $\alpha'' = 0$, and $\omega''$ is a Ricci-flat metric on $X''$. From \cite[Part.~5.1]{Gue25}, we have $\Hol_q(X''_\reg, \omega''_\reg) = 1, \SU(n)$ or $\Sp (n/2)$, a contradiction. We deduce that $f^*\Omega^1_Y = 0$, so that $Y$ has dimension zero, and the MRC-fibration $f$ is constant. Hence, $X''$ is RC.\\

Let us assume that $(X_\reg, \omega_\reg)$ is locally symmetric and $\widetilde{q}(X) = 0$. The restricted holonomy action being irreducible, with the same arguments as the case $\Hol_p(X_\reg, \omega_\reg) = \U(n)$, we deduce that $f$ is an isomorphism and $\omega''$ is Ricci-flat. We deduce that $(X''_\reg, \omega''_\reg)$ is flat, hence there exists a quasi-\'{e}tale map from an abelian variety $A' \to X''$, which is a contradiction since $\widetilde{q}(X) = 0$.\\

We show $(ii) \Longrightarrow (i)$.
If $X$ is HRC, then we have $\widetilde{q}(X) = 0$ by \cite[Thm.~5.1]{GKKP11}. Assume that $(X_\reg, \omega_\reg)$ is not locally symmetric. Then, by Berger--Simons classification, we have $H_X^\circ = \U(n), \SU(n)$ or $\Sp(n/2)$.\\

By contradiction, let us assume that $H_X^\circ = \SU(n)$ or $\Sp(n/2)$. Then, up to passing to the holonomy cover\footnote{From the correspondence \cite[Thm.~23]{CGG24}, the kernel of the morphism $\pi_1(X_\reg) \to H_X/H_X^\circ$ corresponds to a finite quasi-\'{e}tale cover $X' \to X$, called the \emph{holonomy cover}. On this cover, we have $H_{X'} = H_{X'}^\circ = H_X^\circ$. } we can assume that $H_X = H_X^\circ$. We deduce from the holonomy that $\omega_\reg$ is Ricci-flat. It follows that $\alpha = 0$ over $X_\reg$, hence over $X$ since it is a smooth $(1,1)$-form. Let us denote $\mathrm{PH}_X$ the sheaf of pluriharmonic functions over $X$. The form $\alpha$ yields a class in $H^1(X, \mathrm{PH}_X)$, hence in $H^2(X, \R)$, since we have a morphism:
$$
c_1' : \mathrm{Pic}(X) \longrightarrow H^1(X, \mathrm{PH}_X) \longrightarrow H^2(X, \R).
$$
One can write through local embeddings covering of $X$: $\alpha =: \ddc \varphi_i$, and the family $(\varphi_j - \varphi_i)_{ij}$ produces an element of $Z^1(X, \mathrm{PH}_X)$ as a \v{C}ech cocycle. Since $\alpha = 0$, we get that $\varphi_i \in \mathrm{PH}_X(U_i)$, hence, the associated class in $H^1(X, \mathrm{PH}_X)$ is zero. But the corresponding class in $H^2(X, \R)$ is $c_1(-K_X)$. Indeed, if we pick an open covering and cocycles $(g_{ij})$ of $-K_X$, the class $c_1(-K_X) \in H^2(X, \R)$ is obtained by writing locally $g_{ij} = e^{2\pi \textbf{i}\psi_{ij}}$, and to consider $c_{ijk} := \psi_{ij} + \psi_{jk} - \psi_{ik}$. To construct the connecting map from $H^1(X, \mathrm{PH}_X) \to H^2(X, \R)$, we consider a Cech cocycle $(\varphi_{ij})$, and we write $\varphi_{ij} = \mathrm{Im}(\Psi'_{ij})$, with $\Psi'_{ij}$ holomorphic, and we construct $c'_{ijk} := \Psi'_{ij} + \Psi'_{jk} - \Psi'_{ik}$. Since we have the equality of potentials $\varphi_i - \varphi_j = \log |g_{ij}|^2$, we can choose $\Psi'_{ij} = -4\pi \Psi_{ij}$. This shows that $c_1' = -4\pi c_1$. It follows that $c_1(-K_X) = 0$ and, $X$ has numerically trivial determinant.\\ 

Up to passing to some further quasi-\'{e}tale cover one can assume that $X$ has canonical singularities\footnote{We even deduce from \cite[Prop.~12.10]{GGK19} and the fact that the holonomy is $\SU(n)$ or $\Sp(n/2)$ that $X$ is ICY or IHS, but we will not use it.}, by \cite[Theorem~7.1]{GGK19}. Let $\widehat{X} \to X$ be a resolution of singularities. Since $X$ is assumed HRC, we deduce that $\widehat{X}$ is RC, hence uniruled, and there exists a fibration $\widehat{X} \to Y$ with fiber $\mathbb{P}^1$. If $C$ is a general fiber, $K_{\widehat{X}} \cdot C = \deg(K_C) < 0$, but:
$$
K_{\widehat{X}} \cdot C = \pi^*K_X \cdot C + \sum_i a_i E_i \cdot C \geq 0,
$$
because $\pi^*K_X \cdot C = 0$, and we can choose $C$ general so that $C \cdot E_i \geq 0$ and $a_i \geq 0$ because $X$ has canonical singularities, a contradiction.
\end{proof}

\section{Structure of klt varieties with semi-positive anticanonical bundle}

\begin{theo}[= Theorem~\ref{Decomposition}]
	Let $X$ be a projective variety with canonical singularities such that $-K_X$ is semipositive. Then, there exists a finite quasi-\'{e}tale cover $f : X' \to X$ such that the Albanese map:
	$$
	\mathrm{alb}_{X'} : X' \to \mathrm{Alb}(X')
	$$
	is locally constant and its fiber is a product $\prod_{j} Y_j \times \prod_{k} S_k \times  \prod_{l} Z_l$, where $Y_j$ are ICY, $S_k$ are IHS and $Z_l$ are HRC.
\end{theo}

\begin{proof}
Up to passing to a finite quasi-\'{e}tale cover, let us assume that $X$ is maximally quasi-\'{e}tale with $q(X) = \widetilde{q}(X) =: q$ and the holonomy of $(X_\reg, \omega_\reg)$ is connected.
We have the Albanese morphism:
$$
a : X \longrightarrow A
$$
that is locally constant, and its general fiber is $X_a$. It is of the form $F \times Y'$ where $F$ is rationally connected and $Y'$ is a product of ICY and IHS varieties. The fiber $X_a$ is hence a klt projective variety such that $-K_{X_a}$ is semi-positive. From the decomposition:
$$
\mathcal{T}_X = a^*\mathcal{T}_A \oplus \bigoplus_{j \in J} \mathcal{F}_j,
$$
we deduce that:
$$
\mathcal{T}_{X_a} = \bigoplus_{j \in J} \mathcal{F}_{j|X_a},
$$
and this is the decomposition of the action of $H_{X_a}^\circ$. In particular, at a point $p \in X_a$, the subspace of fixed points is trivial. Hence, we deduce from Theorem~\ref{Identification euclidean factor augmented irregularity} that $\widetilde{q}(X_a) = 0$. Let us investigate this case.

\begin{lemm}\label{Decomposition for vanishing augmented irregularity}
	Let $X$ be a projective variety with canonical singularities such that $-K_X$ is semipositive and $\widetilde{q}(X) = 0$. Then, there exists a finite quasi-\'{e}tale cover $X' \to X$ such that:
	$$
	X' = \prod_{j} Y_j \times \prod_{k} S_k \times  \prod_{l} Z_l,
	$$
	where $Y_j$ are ICY, $S_k$ are IHS and $Z_l$ are HRC.
\end{lemm}

\begin{proof}[Proof of the lemma]
In order to apply \cite[Thm.~25]{DGP24}, we show that the holonomy bundles $\mathcal{F}_{j}$ have algebraic leaves.
From the positivity assumption on $-K_X$ and the fact that the decomposition 
$$\mathcal{T}_X = \bigoplus_j \mathcal{F}_j$$ is orthogonal, we get that $\mu_{\omega}(\mathcal{F}_{j}) \geq 0$. Note that $\mathcal{F}_0 = 0$ here as a consequence of $\widetilde{q}(X) = 0$ and Theorem~\ref{Identification euclidean factor augmented irregularity}.\\

Without loss of generality, we assume that $X$ is smooth in codimension 2, something that can be achieved by considering a $\Q$-factorial terminalization $\beta : X_{\mathrm{ter}} \to X$ whose existence is guaranteed by \cite[Cor.~1.4.3]{BCHM10}. It is a small birational modification such that $X_{\mathrm{ter}}$ has terminal singularities, hence it is smooth in codimension 2. Because $\beta$ is a crepant morphism, we deduce that $-K_{X_\mathrm{ter}}$ is semipositive by pulling back the non-negatively curved hermitian metric on $-K_X$ to $-K_{X_\mathrm{ter}}$ through $\beta$. Moreover, we have $\widetilde{q}(X_{\mathrm{ter}}) = 0$: this can be derived from the arguments \cite[Lemma~4.4]{Dru18} and using that klt singularities are rational. Now, if there exists a finite quasi-\'{e}tale morphism $\widetilde{X_{\mathrm{ter}}} \to X_\mathrm{ter}$ and a decomposition of $\widetilde{X_{\mathrm{ter}}}$, then the Stein factorization $\widetilde{X} \to X$ of the composition $\widetilde{X_{\mathrm{ter}}} \to X_{\mathrm{ter}} \to X$ is quasi-\'{e}tale and it splits into a product of varieties thanks to \cite[Lemma~4.6]{Dru18}.\\

\begin{itemize}
	\item Case 1 : $\mu_{\omega}(\mathcal{F}_{j}) > 0$
	
	We argue similarly to \cite[Claim.~27]{DGP24}. Let $m$ be a positive integer such that $mA$ is very ample. Let us choose $C$ a general curve in $|mA|$ such that $C \subseteq X_\reg$ and $\mathcal{F}_{j|C}$ is locally free on a neighborhood of $C$. If $m$ is sufficiently large, we have $\mathcal{F}_{j|C}$ that is semistable, hence ample since it has positive slope. Then, it has algebraic leaves thanks to \cite[Fact.~2.1.1]{BM16}.\\
	
	\item Case 2 : $\mu_{\omega}(\mathcal{F}_{j}) = 0$. 
	
	First, the restricted holonomy $G_j^\circ$ acting on $\mathcal{F}_j$ at a point $p \in X_\reg$ is $\U(n_j)$, $\SU(n_j)$ or $\Sp(n_j/2)$. Indeed, let us show that $\mathcal{F}_j$ is not of symmetric type. By contradiction, if $\mathcal{F}_j$ is of symmetric type, the local decomposition into a product of K\"{a}hler manifolds \eqref{Local decomposition of Kahler manifold} show that the associated hermitian manifold $(U_j, \omega_j)$ is symmetric and Ricci-flat\footnote{It is Ricci-flat since we have from local decompositions of neighborhood $U$ of $p$ into a product of manifolds $(U_0, \omega_{\mathrm{euc}}) \times \prod_j (U_j, \omega_j)$ that $\textbf{i}\Theta(\mathcal{F}^\vee_j, h_j^\vee) \wedge \omega^{n-1} = -\overline{\sharp \Ric \omega_j} \omega^{n} = -\mathrm{pr}_{\mathcal{F}^\vee_j}(\overline{\sharp \alpha}_{\mathcal{F}^\vee_j}) \omega^n = 0$ by Proposition~\ref{Polystability of vanishing slopes} second part of $(i)$. It follows that $G_j^\circ$ is $\SU(n_j)$ or $\Sp(n_j/2)$.}, hence it is flat contradicting the non-triviality of the action of $G_j^\circ$.
	
	Let us show that $\Sym^{[m]}\mathcal{F}_j$ is stable. First, one can show that it is polystable. In order to apply point $(iii)$ of Proposition~\ref{Polystability of vanishing slopes}, it is sufficient to show that $c_1(\Sym^{[m]}\mathcal{F}_j) \cdot \left[ \omega \right]^{n-1} = 0$. 
	One can construct two subsheaves of $\Omega^1_{\widehat{X}}$: the first one is $S_1 := \Sym^{[m]} \widehat{\mathcal{F}_j}$, and the second one is $S_2 := \widehat{\Sym^m \mathcal{F}_j}$, where $\widehat{\bullet}$ means the saturation of the sheaf (in $\Omega^1_{\widehat{X}}$ for $\mathcal{F}_j$ and in $\Sym^m \Omega^1_{\widehat{X}}$ for $\Sym^m \mathcal{F}_j$). 
	Both sheaves coincide over $\widehat{X} \backslash E$, hence there exists numbers $k_i$ such that $c_1(S_2) = c_1(S_1) + \sum_i k_i c_1(E_i) $. There exists a big open subset $\widehat{U'}$ such that $S_1$ is a subbundle of $\Sym^m\Omega^1_{\widehat{X}}$. We have by \cite[Proof of $(\ast\ast)$, p.166]{Kob80} that:
	\begin{align}
	c_1(S_1) \cdot \left[ \omega_{t, \varepsilon}\right]^{n-1} 
	& = \int_{\widehat{U'}} c_1(\Sym^m \widehat{\mathcal{F}_j}, \Sym^m h_{t, \varepsilon}^\vee) \wedge \omega_{t, \varepsilon}^{n-1} \\
	& \underset{t, \varepsilon \to 0}{\longrightarrow} \binom{n_j + m -1}{n_j} c_1(\widehat{\mathcal{F}_j}) \cdot \left[ \pi^* \omega \right]^{n-1} = 0.
	\end{align}
	Hence, we have $c_1(S_1) \cdot \left[ \pi^* \omega \right]^{n-1} = 0$. Since $E_i$ are $\pi$-exceptional, we have:
	\begin{align}
	    0 = c_1(S_1) \cdot \left[ \pi^* \omega \right]^{n-1} = c_1(S_2) \cdot \left[ \pi^* \omega \right]^{n-1}
	\end{align}
	and by the projection formula, we deduce:
	\begin{align}
	    0 = c_1(S_2) \cdot \left[ \pi^* \omega \right]^{n-1} = c_1(\Sym^{[m]} \mathcal{F}_j) \cdot \left[ \omega\right]^{n-1}.
	\end{align}
	We conclude that $\Sym^{[m]} \mathcal{F}_j$ is a proper subsheaf of $\Sym^{[m]}\Omega^{[1]}_{X}$ with vanishing slope. By point $(iii)$ of Proposition~\ref{Polystability of vanishing slopes}, it is polystable and parallel with respect to $\omega$. By contradiction, if it is not stable, one can consider a destabilizing subsheaf $\mathcal{G}$. Thanks to Proposition~\ref{Polystability of vanishing slopes}, it is parallel. Let $p \in X_\reg$ be a point where the inclusion $\mathcal{G} \hookrightarrow \Sym^{[m]}\mathcal{F}_j$ is an inclusion of bundles. Then, there is an irreducible action of $G_j^\circ = \U(n_j), \SU(n_j)$ or $\Sp(n_j/2)$ on $\mathcal{F}_{j,p}$ such that the induced action on $\Sym^m \mathcal{F}_{j,p}$ remains irreducible. Because $\mathcal{G}$ is parallel, it stabilizes the action of $G_j^\circ$, contradicting the irreducibility of the action on $\Sym^m\mathcal{F}_{j,p}$. We deduce that $\Sym^{[m]}\mathcal{F}_j$ is stable with respect to $\left[ \omega \right]$. 
	
	If $\mathcal{F}_j$ is not pseudo-effective, then it has algebraic leaves by \cite[Prop.~8.4]{Dru18}. If it is pseudo-effective, we are in position to apply \cite[Thm.~1.1]{HP19}, and we deduce that:
	$$
	c_1(\mathcal{F}_j) \cdot A^{n-1} = c_2(\mathcal{F}_j) \cdot A^{n-2} = 0.
	$$
	Thanks to \cite[Thm.~1.4]{LT18}, the subbundle $\mathcal{F}_{j|X_\reg}$ is flat and it is given by a representation $\pi_1(X_{\reg}) \to \GL(n_j, \C)$. Up to passing to the maximally quasi-\'{e}tale cover, we can assume that $X$ is maximally quasi-\'{e}tale, so that $\mathcal{F}_{j}$ is flat since this representation factors through $\pi_1(X) \to \GL(n_j, \C)$. The same arguments as Step 2 of the proof of Theorem~\ref{Identification euclidean factor augmented irregularity} show that it contradicts $\widetilde{q}(X) = 0$.
\end{itemize}
\end{proof}

From this lemma, we deduce that there exists a finite quasi-\'{e}tale cover $\widetilde{X_a} \to X_a$ on which $\widetilde{X_a}$ splits as a product $\prod_{j} Y_j \times \prod_{k} S_k \times  \prod_{l} Z_l$ of ICY, IHS and HRC varieties. Secondly, the universal cover of $A$ is $\C^q \to A$, and from the local constancy of the MRC-fibration, we deduce that that $X \underset{A}{\times} \C^q = X_a \times \C^q$, and the projection $X_a \times \C^q \to X$ is \'{e}tale. If we combine both constructions, there exists a quasi-\'{e}tale (not finite) map $\widetilde{X_a} \times \C^q \to X$. Since it is quasi-\'{e}tale, the composition corresponds to a subgroup $G \leq \pi_1(X_\reg) =: \pi^X$, thanks to \cite[Thm.~23]{CGG24}. It is of finite index in $\pi_1(X_{a \reg}) =: \pi^{X_a}$. We show the following lemma.

\begin{lemm}\label{Group argument}
	With notation as above, there exists a finite-index subgroup $H$ of $\pi^X$ such that $H \cap \pi^{X_a} = G$.
\end{lemm}

\begin{proof}[Proof of the lemma]
First, we consider the group $N := \bigcap_{\gamma \in \pi^X} \gamma G \gamma^{-1}$. It is of finite-index\footnote{This comes from the fact that $\gamma G \gamma^{-1}$ is a subgroup of finite-index $\left[ \pi^{X_a} : G \right]$ in $\pi^{X_a}$. Since this latter group is finitely generated, the number of subgroups of index $\left[ \pi^{X_a} : G \right]$ is finite.} in $\pi^{X_a}$ and it is normal in $\pi^X$. In the following, for any group $N \subseteq K \subseteq \pi^X$, we denote $\overline{K} := K/N$. The problem is to construct a subgroup of finite index $\overline{H}$ of $\overline{\pi^X}$ such that $\overline{\pi^{X_a}} \cap \overline{H} = \overline{G}$ in the case where we have an extension:
$$
1 \longrightarrow \overline{\pi^{X_a}} \longrightarrow \overline{\pi^X} \overset{\rho}{\longrightarrow} \Z^{2q} \longrightarrow 1
$$
where the group $\overline{\pi^{X_a}}$ is finite.\\

As a consequence of \cite[Lemma~2.3]{BF24}, the group $\overline{\pi^X}$ is virtually abelian, hence residually finite. This ensures the existence of a finite-index subgroup $\overline{A}$ such that $\overline{A} \cap \overline{\pi^{X_a}} = 1$. Then, the restriction of $\rho$ to $A$ provides an isomorphism onto a finite-index subgroup $S \subseteq \Z^{2q}$. One can consider $\overline{K} := \ker (\overline{\pi^{X}} \overset{\rho}{\longrightarrow} \Z^{2q} \to \Z^{2q}/S)$. This group contains $\overline{\pi^{X_a}}$ and $\overline{A}$, and the morphism $\rho$ provides a splitting:
$$
	\xymatrix{
	1 \ar[r] & \overline{\pi^{X_a}} \ar[r] & \overline{K} \ar[r]^{\rho} & S  \ar[r] \ar@/^1pc/[l] & 1\\  
	}.
$$ 
Up to replace $\overline{K}$ by the normalizer $\overline{N}$ of $\overline{G}$ in $\overline{K}$ that is of finite-index and $\overline{A}$ by $\rho^{-1}(\rho(\overline{N}))$, and $S$ by $\rho(\overline{N})$, we can assume moreover that $\overline{G}$ is normal in $\overline{K}$. We get that the group $\overline{K}$ is a semi-direct product $\overline{K} = \overline{\pi^{X_a}}\,\, \overline{A}$, and the set $\overline{H} = \overline{G}\,\overline{A}$ is a subgroup of $\overline{K}$ of finite index\footnote{This comes from the inclusions $\overline{A} \subseteq \overline{H} \subseteq \overline{K}$.}. Moreover, the uniqueness of the decomposition in the semi-direct product guarantees that $\overline{H} \cap \overline{\pi^{X_a}} = \overline{G}$. This group $\overline{H}$ satisfies the sought properties.
\end{proof}

The existence of the group $H$ from the previous lemma ensures the existence of a finite quasi-\'{e}tale (hence \'{e}tale because $X$ is maximally quasi-\'{e}tale) cover $f : X' \to X$ from the correspondence \cite[Thm.~23]{CGG24}. To sum up the situation, we have the following tower of quasi-\'{e}tale maps:
$$
	\xymatrix{
	 & \widetilde{X_a} \times \C^q \ar[d]\ar[ldd] \\
	 & X_a \times \C^q \ar[d] \\
	X' \ar[r] & X
	}.
$$
This yields at the level of Albanese morphisms a commutative diagram:
$$
	\xymatrix{
	X' \ar[d] \ar[r] & \mathrm{Alb}(X') \ar[d]^\beta \\
	X \ar[r] & \mathrm{Alb}(X)
	}.
$$
Since the morphism $X' \to X \to \mathrm{Alb}(X)$ and $X' \to \mathrm{Alb}(X')$ are surjective, so is the morphism $\beta : \mathrm{Alb}(X') \to \mathrm{Alb}(X)$. Moreover, both abelian varieties have same dimension equal to $q$. Hence it is an \'{e}tale morphism\footnote{Indeed, such a morphism induces a surjective morphism between the universal cover $\widetilde{\beta} : \C^q \to \C^q$ sending the lattice $\Lambda'$ (such that $\mathrm{Alb}(X') = \C^q/\Lambda'$) to $\Lambda$ (such that $\mathrm{Alb}(X) = \C^q/\Lambda$). Hence, it follows that for any $\lambda' \in \Lambda'$, we have $\widetilde{\beta}( \cdot + \lambda') - \widetilde{\beta}$ that is constant. We conclude that the derivatives are holomorphic elliptic functions, so that they are constant. Hence, there exists a matrix $A$ such that $\widetilde{\beta}(z) = Az$ which is invertible since it is a surjective endomorphism. We deduce that the derivative of $\beta$ is invertible at any point: the morphism $\beta$ is \'{e}tale.}. The Albanese morphism associated to $X'$ being locally trivial, we denote by $X'_a$ the general fiber of this fibration. Hence, from Serre long exact sequences associated to both Albanese fibrations, we have:
$$
\xymatrix{
	1 \ar[r] & G = \pi_1(X'_{a \reg}) \ar@{^{(}.>}[d] \ar[r] & H = \pi_1(X'_\reg) \ar[r] \ar@{^{(}->}[d] & \pi_1(\mathrm{Alb}(X')) \ar[r] \ar@{^{(}->}[d] & 1\\
	1 \ar[r] & \pi_1(X_{a \reg})  \ar[r] & \pi_1(X_\reg) \ar[r] & \pi_1(\mathrm{Alb}(X)) \ar[r] & 1\\
	}.
$$ 
We deduce that a fiber of $\mathrm{alb}_{X'}$ is sent on a fiber on $\mathrm{alb}_X$ by the morphism $X' \to X$, in other words we have a quasi-\'{e}tale morphism $X'_a \to X_a$. The corresponding subgroup in $\pi_1(X_{a \reg})$ is $G$, providing an isomorphism of quasi-\'{e}tale coverings 
$$X'_a \simeq \widetilde{X_a} = \prod_{j} Y_j \times \prod_{k} S_k \times  \prod_{l} Z_l.$$ 
\end{proof}

\section{Fundamental groups}

The aim of this section is to show the following result.

\begin{theo}[= Theorem \ref{Virtually abelian}]
	Let $X$ be a projective variety with canonical singularities such that $-K_X$ is semipositive. Then, $\pi_1(X_\reg)$ is virtually abelian.
\end{theo}

\begin{proof}
	Up to picking some finite \'{e}tale covers and using Theorem~\ref{Decomposition}, we can assume that the Albanese morphism $a : X \to A$ is locally constant and we have an isomorphism of the fiber $X_a = \prod_{j} Y_j \times \prod_{k} S_k \times  \prod_{l} Z_l$, where $Y_j$ are ICY, $S_k$ are IHS and $Z_l$ are HRC. We have a short exact sequence:
	$$
	0 \longrightarrow \pi_1(X_{a \reg}) \longrightarrow \pi_1(X_\reg) \longrightarrow \pi_1(A)	 \longrightarrow 0.
	$$
	Moreover, we have  $\pi_1(X_{a \reg}) = \prod_{j} \pi_1(Y_{j \reg}) \times \prod_{k} \pi_1(S_{k \reg}) \times  \prod_{l} \pi_1(Z_{l \reg})$. Thanks to \cite[Theorem~1.3]{FGSW26}, the factors $ \pi_1(Y_{j \reg})$ and $\pi_1(S_{k \reg})$ are finite. We show that the remaining factors are finite.

\begin{lemm}
	Let $X$ be a klt projective variety such that $-K_X$ is semipositive. Assume that $X$ is HRC and the holonomy at a point $p \in X_\reg$ acts irreducibly. Then, $\pi_1(X_\reg)$ is finite. 
\end{lemm}

\begin{proof}
	Using \cite[Thm.~1.6]{JLR26}, it suffices to show that $X$ admits a 1-free curve, namely a morphism $s : C \to X_\reg$ such that every positive rank quotient of $s^*\mathcal{T}_X$ has slope greater than $2g(C) + 1$. Since $-K_X$ is semipositive (hence nef), the existence of such curves is related to some positivity property of the tangent bundle guaranteed by \cite[Thm.~3.4]{JLR26}. More precisely, the variety $X$ admits a 1-free curve if there is no positive rank quotient of $\mathcal{T}_X$ with numerically trivial first Chern class. Let us show that this property holds.\\
	
	By contradiction, let $\mathcal{Q}$ be a positive rank quotient of $\mathcal{T}_X$ such that $c_1(\mathcal{Q}) = 0$. Hence, by taking the dual functor which is left exact, $\mathcal{Q}^\vee$ is a subsheaf of $\Omega^{[1]}_X$, such that $c_1(\mathcal{Q}^\vee) = 0$. Thanks to Proposition~\ref{Bochner-Kobayashi property}, the sheaf $\mathcal{Q}^\vee$ is parallel with respect to the connection induced by $\omega$, hence at a general point $p \in X_\reg$, it yields a subrepresentation of the action of the holonomy. Hence, since $\mathcal{Q}^\vee$ is a reflexive subsheaf of $\Omega^{[1]}_X$ and they have same rank, they are isomorphic. Hence, we get that $\mu_{H_0}(\Omega^{[1]}_X) = 0$. We deduce that $\omega$ is a Ricci-flat metric by arguing similarly to the proof of the implication $(ii) \Longrightarrow (i)$ of Lemma~\ref{Link holonomy HRC}, and its holonomy at a general point is $\SU(n)$ or $\Sp(n/2)$, contradicting Lemma~\ref{Link holonomy HRC}.
\end{proof}

The arguments given in the proof of Lemma~\ref{Group argument} show that $\pi_1(X_\reg)$ is residually finite. We deduce that there exists a subgroup $H$ of finite-index such that $H \cap \pi_1(X_{a \reg}) = \lbrace 1 \rbrace$. This indicates that the restriction to $H$ of the morphism $\pi_1(X_\reg) \to \pi_1(A) \simeq \Z^{2q}$ is injective. Hence, $H$ is an  abelian subgroup of $\pi_1(X_\reg)$ with finite index.
\end{proof}

\printbibliography

\end{document}